\documentclass[a4paper, 11pt]{article} 

\usepackage{amsmath, amscd, amsthm} 
\usepackage[applemac]{inputenc}
\usepackage[all]{xy} 
\usepackage[T1]{fontenc} 
\usepackage{textcomp} 
\usepackage[adobe-utopia]{mathdesign}

\DeclareMathOperator{\re}{Re}

\DeclareMathOperator{\End}{End}
\DeclareMathOperator{\Aut}{Aut}
\DeclareMathOperator{\Hom}{Hom}

\DeclareMathOperator{\tr}{Tr}

\DeclareMathOperator{\U}{U}

\DeclareMathOperator{\GL}{GL}

\DeclareMathOperator{\Ric}{Ric}

\DeclareMathOperator{\Hilb}{Hilb}
\DeclareMathOperator{\FS}{FS}

\newcommand{\R}{\mathbb R}
\newcommand{\C}{\mathbb C}

\newcommand{\B}{\mathcal B}
\newcommand{\D}{\mathcal D}

\newcommand{\diff}{\mathrm{d}}
\newcommand{\del}{\partial\!}
\newcommand{\delb}{\overline{\partial}\!}

\newcommand{\omegaFS}{\omega_{\mathrm{FS}}}

\newcommand{\su}{\mathfrak{su}}

\renewcommand{\P}{\mathbb P}
\renewcommand{\H}{\mathcal H}
\renewcommand{\O}{\mathcal O}
\renewcommand{\u}{\mathfrak{u}}

\theoremstyle{plain}
	\newtheorem{theorem}{Theorem}
	\newtheorem{proposition}[theorem]{Proposition}
	\newtheorem{lemma}[theorem]{Lemma}
	\newtheorem{corollary}[theorem]{Corollary}

\theoremstyle{definition}
	\newtheorem{definition}[theorem]{Definition}
	\newtheorem{remark}[theorem]{Remark}

\theoremstyle{plain}
	\newtheorem*{theorem*}{Theorem}
	\newtheorem*{proposition*}{Proposition}
	\newtheorem*{lemma*}{Lemma}
	\newtheorem*{corollary*}{Corollary}
	\newtheorem*{conjecture*}{Conjecture}
	\newtheorem*{inductive-hypotheses}{Inductive Hypotheses}
	\newtheorem*{base-step}{Base Step}
	\newtheorem*{inductive-step}{Inductive Step}
	\newtheorem*{final-step}{Final Step}
\theoremstyle{definition}
	\newtheorem*{definition*}{Definition}
	\newtheorem*{remark*}{Remark}
	\newtheorem*{remarks*}{Remarks}
	
\makeatletter
\def\blfootnote{\xdef\@thefnmark{}\@footnotetext}
\makeatother

\begin{document}

\title{Quantisation and the Hessian of Mabuchi energy}
\author{Joel Fine}
\date{}

\maketitle

\begin{abstract}
Let $L \to X$ be an ample bundle over a compact complex manifold. Fix a Hermitian metric in $L$ whose curvature defines a Kähler metric on $X$. The Hessian of Mabuchi energy is a fourth-order elliptic operator $\D^*\D$ on functions which arises in the study of scalar curvature. We quantise $\D^*\D$ by the Hessian $P^*_kP_k$ of balancing energy, a function appearing in the study of balanced embeddings. $P^*_kP_k$ is defined on the space of Hermitian endomorphisms of $H^0(X,L^k)$ endowed with the $L^2$-inner-product. We first prove that the leading order term in the asymptotic expansion of $P^*_kP_k$ is $\D^*\D$. We next show that if $\Aut(X,L)/\C^*$ is discrete, then the eigenvalues and eigenspaces of $P^*_kP_k$ converge to those of $\D^*\D$. We also prove convergence of the Hessians in the case of a sequence of balanced embeddings tending to a constant scalar curvature Kähler metric. As consequences of our results we prove that an estimate of Phong--Sturm in \cite{phong-sturm} is sharp and give a negative answer to a question of Donaldson from \cite{donaldson-1}. We also discuss some possible applications to the study of Calabi flow.
\end{abstract}

\tableofcontents

\section{Introduction}

\subsection{Background}
\label{motivation}

This article concerns Donaldson's deep observation that balanced projective embeddings are the quantisation of constant scalar curvature Kähler metrics \cite{donaldson-1}. To set the scene and fix notation we briefly outline the main ideas involved.

Let $L \to X^n$ be an ample line bundle over a compact complex manifold. Write $\H$ for the space of Hermitian metrics in $L$ whose  curvature is positive, i.e.\ giving a Kähler metric in $c_1(L)$. Calabi suggested in \cite{calabi-ekm} that, when one exists, a constant scalar curvature metric should be considered as a canonical representative of the Kähler class $c_1(L)$. 

The definition of balanced projective submanifolds is due to Luo \cite{luo} and Zhang \cite{zhang} (see also related work of Bourguignon, Li and Yau \cite{bourguignon-et-al}). There is an embedding $\mu \colon \C\P^m \to i\u(m+1)$, given by sending a point $p \in \C\P^m$ to the Hermitian endomorphism $\mu(p)$ of $\C^{m+1}$ given by orthogonal projection onto the line corresponding to $p$. Given a complex submanifold $Y^n \subset \C\P^m$ we consider the ``centre of mass of $Y$'', $\bar\mu \in i \u(m+1)$, given by
\[
\bar{\mu}
= \int_Y \mu\, \frac{\omegaFS^n}{n!}.
\]
The submanifold $Y$ is called \emph{balanced} if $\bar{\mu}$ is a multiple of the identity. 

To explain Donaldson's result we need a little more notation. Write $\B_k$ for the space of Hermitian inner-products on $H^0(X,L^k)$. Given $b\in \B_k$ we define a metric $\FS_k(b) \in \H$ as follows: for large $k$ a $b$-orthonormal basis of $H^0(X, L^k)$ defines an embedding $X \to \C\P^{n_k}$ (where $n_k+1 = \dim H^0(X,L^k)$); using this we pull back the Fubini--Study metric from $\O(1)$ to a metric on $L^k$ and then take the $k^{\text{th}}$-root to obtain a metric on $L$ which we denote $\FS_k(b)$. It is positively curved since its curvature is, up to scale, the restriction of the Fubini--Study Kähler form to $X$. This defines a map $\FS_k \colon \B_k \to \H$. 

Donaldson proved that balanced embeddings are the quantisation of constant scalar curvature Kähler metrics in the following sense: 

\begin{theorem}[Donaldson \cite{donaldson-1}]
\label{skd_balanced_converges}
Assume that $\Aut(X,L)/\C^*$ is discrete. If the curvature of $h \in \H$ defines a Kähler metric of constant scalar curvature then for all large $k$ there is a point $b_k \in \B_k$, unique up to multiplication by a constant, which defines a balanced embedding of $X$ in $\C\P^{n_k}$. Moreover, with the appropriate choice of scale for each $b_k$, $\FS_k(b_k) \to h$ in $C^\infty$ as $k \to \infty$.
\end{theorem}

There are other strong connections between the problems of finding balanced embeddings and constant scalar curvature metrics which are suggested by Donaldson's work. Most pertinent for us is the fact that both objects of interest in Theorem \ref{skd_balanced_converges} are the critical points of functions. In the case of balanced embeddings there is for each $k$ a function $F_k \colon \B_k \to \R$, called balancing energy. A point of $\B_k$ defines a balanced embedding if and only if it is a critical point of $F_k$. Similarly there is a function $E \colon \H \to \R$, introduced by Mabuchi in \cite{mabuchi-energy}, whose critical points are exactly the metrics of constant scalar curvature. These functions are central to the study of balanced embeddings on the one hand and constant scalar curvature Kähler metrics on the other. The guiding principle of this article---drawn from \cite{donaldson-1} and the sequel \cite{donaldson-2}---is that \emph{balancing energy is the quantisation of Mabuchi energy}. 

These energy functions have an extremely important property: they are geo\-desically convex. In the case of $F_k$, this means with respect to the natural symmetric metric on $\B_k \cong \GL(n_k+1)/\U(n_k+1)$, where $n_k+1 = \dim H^0(X,L^k)$. (This was observed by Zhang, Phong and Sturm, and Paul, \cite{zhang,phong-sturm2,paul}.) For $E$, this means with respect to the Riemannian metric on $\H$ which was introduced independently by Mabuchi, Semmes  and Donaldson \cite{mabuchi-symmetric,semmes,donaldson-symmetric}. This metric is also symmetric, in the sense that its curvature is covariant constant. The functions $F_k$ and $E$ can both be seen as specific instances of Kemp--Ness functions, which arise in the study of moment maps on Kähler manifolds and which are automatically geodesically convex on the relevant symmetric space. We will say a little on this later in this section, but for more detail on this point of view, see \cite{donaldson-fields,donaldson-1}. 

An obvious approach to finding critical points of such functions is to consider their downward gradient flow. The flow of $F_k$ is called \emph{balancing flow}, whilst that of $E$ is called \emph{Calabi flow} and first appeared in Calabi's seminal article \cite{calabi-ekm}. Inspired by Theorem \ref{skd_balanced_converges}, the author proved in \cite{fine} that, provided the balancing flows on each $\B_k$ are started in appropriate places, the induced Fubini--Study metrics converge to Calabi flow for as long as it exists. 

Despite its attractive description as the gradient flow of a geodesically convex function, Calabi flow has proved somewhat intractable from the analytic point of view. This is due in no small part to the fact that it is a fourth-order equation meaning, for example, that arguments directly based on the maximum principle are no longer viable. The results of \cite{fine} suggest an alternative, geometric, approach to understanding Calabi flow, namely via the asymptotics of balancing flow. This article is, in part, a first step in this direction. Our results can be paraphrased by the statement that \emph{the Hessians of balancing energy converge to the Hessian of Mabuchi energy}. Precise sense is given to this in Theorems \ref{asymptotics-hessians}, \ref{eigenvalue-theorem} and \ref{eigenvector-theorem}, along with Theorem \ref{convergence_for_balanced}.

We close this discussion of the background material with a geometric interpretation of this convergence, coming from the general moment-map picture. Suppose a Lie group $K$ acts on a Kähler manifold $Z$ by holomorphic isometries, that the action admits a moment map and that the action extends to the complexification $G$ of $K$. Given a $G$-orbit $G \cdot x \subset Z$, the restriction of the Kähler metric from $Z$ descends, by $K$-invariance, to a (\emph{non}-symmetric) Riemannian metric on the quotient $G/K$. This metric is given precisely by the Hessian of the Kempf--Ness function. In the situation under consideration in this article, we are essentially proving that the metrics induced on the $\B_k$ in this fashion converge to that induced on $\H$ by the moment-map geometry of the scalar curvature. From this point of view our results could be compared to those of Phong and Sturm or Chen and Sun \cite{phong-sturm-geodesics,chen-sun} who investigate how the \emph{symmetric} metrics on $\B_k$ converge to the $L^2$-metric on $\H$

The layout of this article is as follows. In \S\ref{summary} we give precise statements of our main results. In \S\ref{applications} we give some immediate consequences and describe some possible applications of our results to the study of Calabi flow. In \S\ref{asymptotic_results} we outline the technical tools needed in our proofs as well as provide more references to the many prior works in these areas. The remainder of the article is devoted to explaining the proofs themselves.

\subsection{Overview of results}
\label{summary}

We now describe our main results. Let $L \to X^n$ be an ample line bundle over a compact complex manifold. Fix a Hermitian metric $h$ in $L$ whose curvature defines a Kähler metric $\omega$ on~$X$. Given these data, there is a fourth-order non-negative elliptic operator 
\[
\D^*\D \colon C^\infty(X, \R) \to C^\infty(X, \R)
\] 
defined as follows. Given a function $f$, write $v_f$ for the Hamiltonian vector field associated to $f$ via $\omega$. Define $\D \colon C^\infty(X, \R) \to \Omega^{0,1}(TX)$ by $\D(f) = \delb v_f$. Write $\D^*$ for its $L^2$-adjoint. The operator $\D^*\D$ is the Hessian of Mabuchi energy, alluded to above in \S\ref{motivation}. The fact that this operator genuinely is the Hessian of Mabuchi energy is essentially equivalent to the fact that ``scalar curvature is a moment map,'' and is proven in \cite{donaldson-fields}.

We will ``quantise'' $\D^*\D$ as follows. For each large integer $k$ we replace the infinite dimensional space $C^\infty(X, \R)$ by the space $i\u(n_k+1)$ of Hermitian matrices and we replace $\D^*\D$ by an endomorphism of $i\u(n_k+1)$ of the form $P^*_kP_k$, namely the Hessian of balancing energy. Our results describe how $P^*_kP_k$ converges to $\D^*\D$ in the limit $k \to \infty$. 

It is important to remark that $P_k^*P_k$ is \emph{not} produced from $\D^*\D$ via some general quantisation scheme, for then it would be a standard result that it converges to $\D^*\D$ as $k \to \infty$. Rather, the Hessian of balancing energy $P^*_kP_k$ is defined purely in terms of projective geometry. Its candidacy as a quantisation of $\D^*\D$ comes from the mantra that balancing energy is the quantisation of Mabuchi energy and the challenge is to prove that this special choice of quantisation has the required limit. 

To describe $P_k^*P_k$, let $Y \subset \C\P^m$ be a complex submanifold. Fix a choice of Hermitian inner-product on $\C^{m+1}$ and hence a Fubini--Study metric on $\C\P^m$. The normal bundle $N \to Y$ is the orthogonal complement of $TY \subset T\C\P^m|_Y$. Given a Hermitian endomorphism $A$ of $\C^{m+1}$, write $\xi_A$ for the corresponding holomorphic vector field on~$\C\P^m$. Projecting $\xi_A|_Y$ to $N$ defines a map $P \colon i\u(m+1) \to C^\infty(N)$. Now the Fubini--Study metric restricted to $N$ and the induced volume-form on $Y$ define an $L^2$-inner-product on $C^\infty(N)$; meanwhile, the Kill\-ing form $\tr(AB)$ gives an inner-product on $i\u(m+1)$. Define the adjoint map $P^* \colon C^\infty(N) \to i\u(m+1)$ with respect to these inner-products. Taking the composition, we obtain an endomorphism of the space of Hermitian matrices:
\[
P^*P \colon i\u(m+1) \to i\u(m+1).
\] 
That this operator is the Hessian of balancing energy is implicit in \cite{donaldson-1} and is computed directly in different ways in \cite{phong-sturm2,phong-sturm,fine}. It plays a central role in the study of balanced embeddings and in particular in Donaldson's proof of Theorem \ref{skd_balanced_converges}. 

To quantise $\D^*\D$, we apply this construction to the sequence of Kodaira embeddings corresponding to increasing powers of the ample bundle $L \to X$. Let $H^0(X, L^k)$ denote the space of holomorphic sections of $L^k$. The Hermitian metric $h$ and associated Kähler form $\omega$ define a Hermitian $L^2$-inner-product on $H^0(X, L^k)$. For sufficiently large $k$, we pick an orthonormal basis of $H^0(X, L^k)$ and hence an embedding $X \to \C\P^{n_k}$, where $\dim(H^0(X, L^k)) = n_k+1$. Define the endomorphism $P_k^*P_k$ of $i\u(n_k+1)$ by taking $Y$ in the previous paragraph to be the image of $X \to \C\P^{n_k}$.

As mentioned above, our results can be paraphrased by the heuristic statement that the Hessians $P_k^*P_k$ of balancing energy converge to the Hessian $\D^*\D$ of Mabuchi energy. To make this concrete, we need some more notation. Given a function $f \colon X \to \R$, we define a sequence of Hermitian matrices $Q_{f,k} \in i\u(n_k+1)$ as follows. Given $h \in \H$ a positively curved metric in $L$, the $L^2$-inner-product defines a Hermitian metric on $H^0(X,L^k)$ and we view this as a map $\Hilb_k \colon \H \to \B_k$ from the space of positively curved bundle metrics to the space of inner-products on $H^0(X,L^k)$. For fixed choice of $h$, the matrix $Q_{f,k}$ is the derivative at $t=0$ of $\Hilb_k$ along the path $e^{4\pi t f}h$ in $\H$.

To describe $Q_{f,k}$ explicity, let  $\{s_\alpha\}$ denote an orthonormal basis of $H^0(X, L^k)$ for the $L^2$-inner-product. With respect to this basis, $Q_{f,k}$ has matrix
\begin{equation}\label{Q-matrix}
\left(Q_{f,k}\right)_{\alpha \beta}
=
\int_X \left(4\pi k f + \Delta f \right) (s_\alpha, s_\beta) \, \frac{\omega^n}{n!}.
\end{equation}
More invariantly, $Q_{f,k}$ is the endomorphism of $H^0(X, L^k)$ given by first multiplying a holomorphic section by $(4\pi k f + \Delta f)$ and then projecting back $L^2$-orthogonally onto the space of holomorphic sections.

We can now state our first main result:

\begin{theorem}\label{asymptotics-hessians}
Under the map $\Hilb_k \colon \H \to \B_k$ the pull-back of the Hessian of balancing energy admits an asymptotic expansion in which the leading order term is the Hessian of Mabuchi energy. More precisely, let $f, g \in C^\infty(X, \R)$. As $k \to \infty$,
\[
\tr \left(  Q_{f,k} P^*_kP_k\left(Q_{g,k}\right) \right)
=
\frac{k^n}{4\pi}
\int_X f\, \D^*\D g\,\frac{\omega^n}{n!}
+
O(k^{n-1}).
\]
Moreover, if $f$ and $g$ vary in a subset of $C^\infty(X)$ which is compact for the $C^\infty$-top\-ology, then this estimate is uniform in $f$ and $g$. Finally, this estimate is uniform when taken over subsets of $\H$ which give rise to uniformly equivalent Kähler metrics lying in a compact set for the $C^\infty$-topology.
\end{theorem}

Given the large body of work on the asymptotics of the  embeddings $X \to \C\P^{n_k}$ (the relevant parts of which are outlined in \S\ref{asymptotic_results}), the proof of Theorem~\ref{asymptotics-hessians} is not difficult, involving just careful manipulation of known expansions. By contrast, the proofs of our other principle results are more substantial.

Our second main result describes the asymptotics of the eigenvalues of $P^*_kP_k$. Write $\lambda_0 \leq \lambda_1 \leq \lambda_2 \leq \cdots$ for the eigenvalues of $\D^*\D$ and $\nu_{0,k} \leq \nu_{1,k} \leq \cdots \leq \nu_{(n_k+1)^2,k}$ for the eigenvalues of $P_k^*P_k$, in each case repeated according to their multiplicity. 

\begin{theorem}\label{eigenvalue-theorem}
Assume that $\Aut(X,L)/\C^*$ is discrete. Then for each $j = 0, 1, 
2, \ldots $
\[
\nu_{j,k} = \frac{\lambda_j}{64\pi^3 k^2}  + O(k^{-3}).
\]
Moreover, for each $j$ this estimate is uniform when the original choice of Kähler metric varies in a family of uniformly equivalent metrics which is compact for the $C^\infty$-topology.
\end{theorem}

We now turn to our final principal result which, put briefly,  says that the eigenspaces of $P^*_kP_k$ converge to those of $\D^*\D$. To give this sense we first need to introduce some more notation. Write $\mu \colon \C\P^m \to i\u(m+1)$ for the map
\begin{equation}
\label{mu}
\mu_{\alpha \beta}
=
\frac{ x_\beta \bar x_\alpha}{4\pi \sum_\gamma|x_\gamma|^2}.
\end{equation}
(This is the same map $\mu$ which appeared in \S\ref{motivation}.) Given a Hermitian matrix $A \in i\u(m+1)$, we write $H(A) \colon \C\P^m \to \R$ for the function
\begin{equation}
\label{H}
H(A) = \tr(A\mu).
\end{equation} 
This is a Hamiltonian for the Killing vector on $\C\P^m$ associated to the skew-Herm\-itian matrix $iA$. For $A \in i\u(m+1)$, by an abuse of notation we also write $H(A)$ for the restriction of this function to $X \subset \C\P^{m}$. We can also view $H(A)$ in terms of the maps $\Hilb_k$ and $\FS_k$, namely $k^{-1}H(A)$ is the derivative of $\FS_k \colon \B_k \to \H$ in the direction $A$ at the point $\Hilb_k(h)$. With this in hand we can now state the result.

\begin{theorem}\label{eigenvector-theorem}
Assume that $\Aut(X,L)/\C^*$ is discrete. Then, when suitably scaled, the images under $H$ of the eigenspaces of $P^*_kP_k$ converge isometrically to the eigen\-spaces of $\D^*\D$. More precisely:
\begin{enumerate}
\item
Fix an integer $r>0$ and let $F_{r,k}\subset i\u(n_k+1)$ denote the span of the first $r+1$-eigenvalues of $P^*_kP_k$. There is a constant $C$ such that for all $A, B \in F_{r,k}$, 
\[
\left| \tr(AB) - 16\pi^2k^n\langle H(A), H(B) \rangle_{L^2(\omega)} \right|
\leq 
Ck^{-1}\tr(A^2)^{1/2}\tr(B^2)^{1/2}
\]
\item
Fix integers $0<p< q$ such that $\lambda_{p-1} < \lambda_p = \lambda_{p+1} = \cdots = \lambda_q < \lambda_{q+1}$. Write $V_{p}$ for the $\lambda_p$-eigenspace of $\D^*\D$ and write $F_{p,q,k}$ for the span of the $\nu_{j,k}$-eigenspaces of $P^*_kP_k$ with $p\leq j \leq q$.

Given $\phi \in V_{p}$, let $A_{\phi,k}$ denote the point $F_{p,q,k}$ with $H(A)$ nearest to $\phi$, as measured in $L^2$. Then
\[
\left\|
H(A_{\phi,k}) - \phi
\right\|^2_{L^2_2(\omega)}
=
O(k^{-1})
\]
and this estimate is uniform in $\phi$ if we require in addition that $\|\phi\|_{L^2} = 1$. 
\end{enumerate}
Finally, these estimates are uniform when the original choice of Kähler metric varies in a family of uniformly equivalent metrics which is compact for the $C^\infty$-topology.
\end{theorem}

We remark that, by Theorem \ref{eigenvalue-theorem}, $\dim F_{p,q,k} = q-p = \dim V_p$ for large $k$ and so this result tells us that, after suitably rescaling, $H$ is asymptotically an isometry from $F_{p,q,k}$ to $V_p$. We should also expand on the uniformity with respect to the Kähler metric in part 2, since different metrics will have, in general, different spectral gaps. What this uniformity means is that given some positive integer $M$, the estimates are uniform in the metric provided we limit our choice of integers $p,q$ to lie in the range $0<p<q<M$.

\subsection{Applications}
\label{applications}

Before turning to the proofs, we first give some consequences of these results, beginning with the convergence of eigenvalues (Theorem \ref{eigenvalue-theorem}). Control of the eigenvalue $\nu_{1,k}$ is critical to Donaldson's proof of Theorem~\ref{skd_balanced_converges}. Under more general circumstances, concerning certain families of projective embeddings of $X$ for each value of $k$ (so-called $R$-bounded embeddings; see Definition \ref{R-bounded}) instead of just the single embedding $\Hilb_k(h)$ we consider here, and still assuming $\Aut(X,L)/\C^*$ is discrete, Donaldson proved the existence of a constant $C>0$ such that $\nu_{1,k} \geq Ck^{-4}$. This was then refined by Phong and Sturm \cite{phong-sturm} to the lower bound $\nu_{1,k} \geq Ck^{-2}$. (See Theorem~\ref{phong-sturm-theorem} later on for a precise statement of this result. Phong and Sturm's proof of this lower bound greatly influenced the arguments in~\S\S\ref{hessians-complex-submfd}--\ref{2ff-asymptotics}.) Since the sequence $\Hilb_k(h)$  has $R$-bounded geometry, we have the following immediate corollary of Theorem~\ref{eigenvalue-theorem}:

\begin{corollary}\label{phong-sturm_sharp}
Phong and Sturm's estimate $\,\nu_{1,k} \geq Ck^{-2}$ is sharp for every polarised complex manifold $L \to X$.
\end{corollary}
\begin{proof}
The kernel of $\D^*\D$ is precisely those Hamiltonian functions which generate holomorphic vector fields lifting to $L$ (see, e.g., \cite{donaldson-1}). The assumption that  $\Aut(X,L)/\C^*$ is discrete thus implies $\ker \D^*\D = \R$. It follows that $\lambda_1 >0$ and so $Ck^{-2} \leq \nu_{1,k} \leq Dk^{-2}$ for constants $D>C>0$.
\end{proof}

Our precise asymptotic description of $\nu_{1,k}$ is also sufficient to give a negative answer to a question of Donaldson, raised in \S4.3 of \cite{donaldson-1}. The question concerns the quantity
\begin{equation}
\label{Lambdak}
\Lambda_k = \max_{A\in i\su(n_k+1)} \frac{\|A\|_{\text{op}}}{\|P_k^*P_kA\|_{\text{op}}}
\end{equation}
where $\|A\|_{\text{op}}$ denotes the operator norm of the matrix $A$, i.e.\ the maximum of the moduli of its eigenvalues, and where the maximum is taken over all $A$ which are trace-free, where $P_k^*P_k$ is injective, so the fraction makes sense. (There is a typo in the relevant equation on page 520 of \cite{donaldson-1}, where the formula given is for $\Lambda^{-1}_k$ and not $\Lambda_k$ as stated.) Donaldson guesses that $\Lambda_k = O(k)$ (again, this is for $R$-bounded projective embeddings, not just the specific sequence $\Hilb_k(h)$ considered in our main results).  As he explains, such a bound would lead to a much simpler proof of the results of \cite{donaldson-1}. Unfortunately, it follows from Theorem \ref{eigenvalue-theorem} that this guess is wrong:

\begin{corollary}\label{skd_wrong}
Assume that $\Aut(X,L)/\C^*$ is discrete. Let $\Lambda_k$ be the quantity defined in (\ref{Lambdak}), where the Hessian $P^*_kP_k$ of balancing energy is evaluated at $\Hilb_k(h)$. Then there is a constant $C>0$ such that $\Lambda_k \geq Ck^2$. Moreover, this holds, no matter what norms one uses on $i\su(n_k+1)$ to define~$\Lambda_k$. 
\end{corollary}
\begin{proof}
Taking $A_k$ to be a $\nu_{1,k}$-eigenvector we see that
\[
\frac{\|A_k\|_{\text{op}}}{\|P_k^*P_kA_k\|_{\text{op}}}
=
\frac{1}{\nu_{1,k}} 
=
\frac{64 \pi^3}{\lambda_1}k^2 + O(k)
\]
It follows that $\Lambda_k \geq Ck^2$ for some $C>0$. The same constant clearly works for any choices of norm.
\end{proof}

We next give an application of Theorems \ref{asymptotics-hessians}, \ref{eigenvalue-theorem} and \ref{eigenvector-theorem} which hinges on the uniformity with respect to the underlying Kähler metric. As stated, these results apply to the specific sequence $\Hilb_k(h) \in \B_k$ of projective metrics. The uniformity enables us to extend them, however, to a different sequence, namely the balanced embeddings coming from Theorem \ref{skd_balanced_converges}. We state this here, and give the proof in \S\ref{Hessians_of_balanced}.

\begin{theorem}\label{convergence_for_balanced}
Assume that $\Aut(X,L)/\C^*$ is discrete and that $h \in \H$ has curvature a constant scalar curvature Kähler metric $\omega_{\mathrm{csc}}$. For all large $k$, write $b_k \in \B_k$ for the balanced metric, whose existence is guaranteed by Theorem \ref{skd_balanced_converges}, scaled so that $\FS(b_k) \to h$ in $C^\infty$. 

In the following, a subscript $k$ denotes an object {\bfseries computed with respect to $b_k$} (as opposed to a sequence of the form $\Hilb_k(h)$ for fixed $h$). Meanwhile, $\D^*\D$ is the operator defined by $\omega_{\mathrm{csc}}$.
\begin{enumerate}
\item
Let $f, g, \in C^\infty(X,\R)$. Then 
\[
 \tr\left(Q_{f,k}P^*_kP_k\left(Q_{g,k}\right)\right)
=
\frac{k^n}{4\pi}\int_X f \D^*\D g\, \frac{\omega^n_{\mathrm{csc}}}{n!}
+
O(k^{n-1})
\]
\item For each $j = 0, 1, 2, \ldots,$ $64 \pi^3k^2\nu_{j,k} \to \lambda_j$.
\item
Fix an integer $r>0$ and let $F_{r,k}\subset i\u(n_k+1)$ denote the span of the first $r+1$-eigenvalues of $P^*_kP_k$. Then there exists a constant $C$ such that for all $A, B \in F_{r,k}$, 
\[
\left| \tr(AB) - 16\pi^2k^n\langle H(A), H(B) \rangle_{L^2(\omega_{\mathrm{csc}})} \right|
\leq Ck^{-1}\tr(A^2)^{1/2}\tr(B^2)^{1/2}
\]
\item
Write $0 < \lambda_1 \leq \lambda_2 \leq \cdots$ for the eigenvalues of $\D^*\D$. Fix integers $0<p< q$ such that $\lambda_{p-1} < \lambda_p = \lambda_{p+1} = \cdots = \lambda_q < \lambda_{q+1}$. Write $V_{p}$ for the $\lambda_p$-eigenspace of $\D^*\D$ and write $F_{p,q,k}$ for the span of the $\nu_{j,k}$-eigenspaces of $P^*_kP_k$ with $p\leq j \leq q$.

Given $\phi \in V_{p}$, let $A_{\phi,k}$ denote the point $F_{p,q,k}$ with $H(A_{\phi,k})$ nearest to $\phi$, as measured in $L^2$. Then $H(A_{\phi,k})$ converges to $\phi$ in $L^2_2$ and this convergence is uniform in $\phi$ if we require in addition that $\|\phi\|_{L^2} = 1$. 
\end{enumerate}
\end{theorem}

This Theorem has a potentially important consequence for the approach mentioned in \S\ref{motivation} to understanding Calabi flow via the asymptotics of balancing flows. Under the same assumptions as in Theorem \ref{convergence_for_balanced} it is natural to expect that for any $\omega \in c_1(L)$, Calabi flow starting at $\omega$ exists for all time and converges to $\omega_{\mathrm{csc}}$. We would like to see if this can be proved via balancing flows. 

Calabi flow is known to exist for short time (see \cite{chen-he}) and so we know from \cite{fine} that for short time at least the balancing flows give metric flows which converge to Calabi flow. Meanwhile, Theorem \ref{skd_balanced_converges} tells us that there are balanced embeddings $b_k$ for all large $k$ which induce metrics which converge to $\omega_{\mathrm{csc}}$. This means that for large $k$ balancing energy is a proper function and so the balancing flow on $\B_k$ converges to $b_k$. So Theorem \ref{skd_balanced_converges} tells us that ``at $t=\infty$'' the balancing flows also converge. The difficulty is to extend this convergence back to finite values of $t$. 

As a first step, one might try to do this infinitesimally: for each balancing flow $\gamma(t)$ on $\B_k$, if we normalise $\gamma'(t)$ to have unit length, we can take a limit as $t \to \infty$ to obtain a ``tangent at infinity'' $v_k \in T_{b_k}\B_k$. Since $\gamma(t)$ is a downward gradient flow, $v_k$ is an eigenvector for the Hessian of balancing energy at $b_k$. Now Theorem \ref{convergence_for_balanced} strongly suggests that the $v_k$ converge; this limit should be the tangent at infinity to Calabi flow.

One might also envisage going further. By a classical result  of Hartman, there is a $C^1$-diffeomorphism from a neighbourhood $U \subset \B_k$ of $b_k$ to a neighbourhood $V \subset T_{b_k}\B_k$ of the origin which identifies balancing flow on $U$ with the linear flow on $V$ defined by the Hessian of balancing energy (see \cite{hartman}). By analysing the asymptotics of these diffeomorphisms one might hope to investigate the long time existence of Calabi flow, with Theorem \ref{convergence_for_balanced} providing the initial step.  

\subsection{Technical background}
\label{asymptotic_results}

In this section we introduce some of the technical tools necessary for the proofs of Theorems \ref{asymptotics-hessians}, \ref{eigenvalue-theorem} and \ref{eigenvector-theorem}. They are the product of the efforts of many people and accurately reflecting the contributions of all of them is a near impossible task, although I have done my best.

We begin with a foundational result, due to Tian \cite{tian}, along with its refinements by Ruan \cite{ruan}, Catlin \cite{catlin}, Lu \cite{lu} and Zelditch \cite{zelditch}. Let $h$ be a positively curved Hermitian metric in $L$, with corresponding Kähler form $\omega$. As above, we consider an $L^2$-orthonormal basis $\{s_\alpha\}$ of $H^0(X, L^k)$, giving a sequence of embeddings $X \to \C\P^{n_k}$. Write $h_k$ and $\omega_k$ for the restriction of the Fubini--Study metric from $\mathcal O(1) \to \C\P^{n_k}$ to $L^k \to X$. Tian's result is that $h_k^{1/k} \to h$ and $\frac{1}{k}\omega_k \to \omega$. (In our earlier notation, $h_k^{1/k} = \FS_k\circ \Hilb_k(h)$.)

More precisely, the Fubini--Study and original metrics are related as follows. Define the \emph{Bergman function} $\rho_k(\omega) \colon X \to \R$ by
\[
\rho_k(\omega) = \sum_{\alpha} | s_\alpha|^2,
\]
where $|\cdot |$ is the norm $h^k$ on sections of $L^k$. Note that $\rho_k$ does not depend on the choice of orthonormal basis, nor does it change when $h$ is multiplied by a constant; it depends solely on $\omega$ as the notation indicates. The original and Fubini--Study metrics are related by:
\[
h_k = \rho_k^{-1} h^k,\qquad
\omega_k = k \omega - \frac{i}{2\pi} \delb\del \log \rho_k.
\]
So the relation between the Fubini--Study and original metrics comes down to the asymptotic behaviour of $\rho_k$. This is given in the following result, due to the work of Catlin \cite{catlin}, Lu \cite{lu}, Ruan \cite{ruan}, Tian \cite{tian} and Zelditch \cite{zelditch}; we state it largely as it appears  in \cite{donaldson-1} (see  Proposition 6 there). For much more information on this topic, including a treatment of the purely symplectic case, see the book \cite{ma-marinescu-book}.

\begin{theorem}\label{asymptotics-Bergman}
~
\begin{enumerate}
\item
For fixed $\omega$ there is an asymptotic expansion as $k \to \infty$,
$$
\rho_k(\omega)
=
b_0(\omega)k^n + b_1(\omega) k^{n-1} + \cdots,
$$
where $n = \dim X$ and where the $b_j(\omega)$ are smooth functions on $X$ which are polynomials in the curvature of $\omega$ and its covariant derivatives.  
\item
In particular, 
$$
b_0(\omega) = 1, \quad b_1(\omega) = \frac{1}{8\pi} S(\omega),
$$
where $S(\omega)$ denotes the scalar curvature of $\omega$.
\item
The expansion holds in $C^\infty$ in that for any $r, M >0$, 
$$
\left\|
\rho_k(\omega) - \sum_{j=0}^M b_j(\omega)k^{n-j}
\right\|_{C^r(X)}
\leq
C_{r, M} k^{n-M -1}
$$
for some constants $C_{r,M}$. 
\item
The constants $C_{r,M}$ can be chosen independently of $\omega$ provided it varies in a family of uniformly equivalent metrics which is compact for the $C^\infty$ topology.
\item
It follows that $\frac{1}{k}\omega_k \to \omega$ and $h_k^{1/k} \to h$ in $C^\infty$.
\end{enumerate}
\end{theorem}

Another important technical tool in our proofs is provided by the Berezin--Toep\-litz quantisation of the algebra of real functions on $X$. This assigns to a function $f \in C^\infty(X,\R)$, the self-adjoint endomorphism $T_{f,k}$ of $H^0(X, L^k)$ defined as follows: first multiply a section by $f$, then use the $L^2$-inner-product to project back to the space of holomorphic sections. $T_{f,k}$ is called the \emph{$k^{\text{th}}$ Toep\-litz operator} of $f$. This was first shown to be a well-defined quantisation with the correct semi-classical limit by Bordemann, Meinrenken and Schlichenmaier \cite{bordemann-et-al}. See also \cite{schlichenmaier}.

Our interest in Toeplitz operators stems from the fact that $Q_{f,k}$ is the $k^\text{th}$-Toeplitz operator of $4\pi k f + \Delta f$. In practice, it is the Toeplitz integral kernels which arise in the proof of Theorem \ref{asymptotics-hessians}. The results we will need are an asymptotic expansion of the restriction to the diagonal of the integral kernels of both $T_{f,k}$ and the composition $T_{f,k}T_{g,k}$. The existence (including in the purely symplectic case) of these expansions is due to Ma and Marinescu; see \cite[Lemma 4.6]{ma-marinescu-paper} (cf.\ also \cite[(4.79)]{ma-marinescu-paper} ,\cite[Lemma 7.2.4 and (7.2.6)]{ma-marinescu-book}). The calculation of the coefficients, which is  essential for our application, was carried out in  \cite{ma-marinescu}.

We now describe the parts of the expansions we will need.
The kernel of $T_{f,k}$ is a section of $L_{1}^{k}\otimes (L_{2}^{k})^{*}$ over $X \times X$, where $L_{1}$ denotes the pull-back of $L$ from the first factor and $L_{2}$ the pull-back from the second factor. Explicitly, if $\{s_\alpha\}$ is an $L^2$-orthonormal basis, the kernel is 
\[
\tilde K_{f,k}(x,y) 
= 
\sum_{\alpha, \beta}
\int_{X} f(z) (s_{\alpha}, s_{\beta})(z) \, 
s_{\beta}(y) \otimes s_{\alpha}^{*}(x) \, \frac{\omega^{n}_{z}}{n!}
\]
where $s^{*}$ is the section of $(L^{k})^{*}$ metric-dual to $s$. On restriction to the diagonal, $L_{1}$ and $L_{2}$ are identical and so $\tilde K_{f,k}(x,x)$ is a function which we write as $K_{f,k}$:
\[
K_{f,k}(x) 
=
\tilde K_{f,k}(x,x) 
=
\sum_{\alpha, \beta}
\int_{X} f(z) (s_{\alpha}, s_{\beta})(z)
(s_{\beta}, s_{\alpha})(x) \, \frac{\omega^{n}_{z}}{n!}
\]

\begin{theorem}[Ma--Marinescu \cite{ma-marinescu}]
\label{mm-first-expansion}
There is an asymptotic expansion
\[
K_{f,k} = q_{f,0}k^{n} + q_{f,1}k^{n-1} + q_{f,2}k^{n-2} + \cdots,
\]
for smooth functions $q_{f,j}$. This holds in $C^{\infty}$ in the same sense as Theorem~\ref{asymptotics-Bergman}. Moreover, there are the following formulae for the first three coefficients:
\begin{eqnarray*}
q_{f,0} 
	&=& 
		f,\\
q_{f,1} 
	&=& 
		\frac{S(\omega)}{8\pi} f - \frac{1}{4\pi} \Delta f,\\
q_{f,2} 
	&=& 
		b_{2}f 
		+
		\frac{1}{32 \pi^{2}} \Delta^{2} f
		-
		\frac{1}{32 \pi^{2}} S(\omega) \Delta f
		+
		\frac{1}{8\pi^{2}} ( \Ric, i \delb\del f),
\end{eqnarray*}
where $b_{2}$ is the coefficient in the expansion of the Bergman function of Theorem~\ref{asymptotics-Bergman} and $\Ric$ is the Ricci form of $\omega$. The expansion is uniform in $f$ varying in a subset of $C^\infty(X)$ which is compact for the $C^\infty$-topology. Finally, it is also uniform in the Kähler metric provided it varies in a family of uniformly equivalent metrics which are compact for the $C^\infty$-topology.
\end{theorem}

The uniformity in $f$ and the metric is not mentioned explicitly in \cite{ma-marinescu}, but it follows immediately, since the estimates in their proofs only involve $f$, the metric and their derivatives. Notice that when $f=1$, $K_{f,k} = \rho_{k}$ and so we can see Theorem~\ref{mm-first-expansion} as a generalisation of Theorem~\ref{asymptotics-Bergman}. In fact, Ma--Marinescu prove a more general result than this as they consider Toeplitz operators for $L^k \otimes E$ where $E$ is a holomorphic vector bundle; we merely state here the version which is sufficient for our needs.

We next turn to the kernel of the composition $T_{f,k}\circ T_{g,k}$,
\[
\tilde K_{f,g,k}(x,y) = \int_{X} \tilde K_{f,k}(x,z) \tilde K_{g,k}(z,y)\, \frac{\omega^{n}_{z}}{n!}.
\]
Restricting to the diagonal gives the function
\[
\tilde K_{f,g,k}(x,x)
=
\sum_{\alpha, \beta, \gamma}
\int_{X\times X}
f(y) g(z)
(s_{\alpha}, s_{\beta})(y)
(s_{\beta}, s_{\gamma})(z)
(s_{\gamma}, s_{\alpha})(x)
\,\frac{\omega^{n}_{y}\wedge \omega^{n}_{z}}{(n!)^{2}}.
\]
We write $K_{f,g,k}(x) = \tilde K_{f,g,k}(x,x)$. Ma and Marinescu also compute the asymptotic expansion of this.

\begin{theorem}[Ma--Marinescu \cite{ma-marinescu}]
\label{mm-second-expansion}
There is an asymptotic expansion
\[
K_{f,g,k} 
= 
q_{f,g,0}k^{n} + q_{f,g,1}k^{n-1} + q_{f,g,2}k^{n-2} + \cdots,
\]
for smooth functions $q_{f,g,j}$. This holds in $C^{\infty}$ in the same sense as Theorem~\ref{asymptotics-Bergman}. Moreover, the first two coefficients satisfy:
\begin{eqnarray*}
q_{f,g,0} 
	&=& 
		fg,\\
q_{f,f,1} 
	&=& 
		\frac{S(\omega)}{8\pi} f^{2} 
		- 
		\frac{1}{2\pi} f\Delta f
		+
		\frac{1}{4\pi}|\diff f|^{2}.
\end{eqnarray*}
The expansion is uniform in $f,g$ varying in a subset of $C^\infty(X)$ which is compact for the $C^\infty$-topology. Finally, it is also uniform in the Kähler metric provided it varies in a family of uniformly equivalent metrics which are compact for the $C^\infty$-topology.
\end{theorem}
Again, Ma and Marinescu prove far more than this: they  work with $L^{k}\otimes E$ for a holomorphic vector bundle~$E$ and also give formulae for $q_{f,g,j}$ for $j = 0,1,2$; we have simply stated here the part of their work which we will need. We also remark that, as for Theorem \ref{mm-first-expansion}, that the uniformity in $f$, $g$ and the metric is not mentioned explicitly in \cite{ma-marinescu}, but it follows directly from their proof since only the derivatives of $f$, $g$ and the metric appear in the estimates.

As a final remark we note that the asymptotic expansions of the kernels described here imply in turn an asymptotic expansion of the product $T_{f,k} \circ T_{g,k}$ itself. This latter expansion was found by different methods by Schlichenmaier \cite{schlichenmaier}.

\subsection{A note on conventions}

We make the convention that when the symbol $C$ appears in an estimate, it denotes a constant which may change from line to line. It is always, however, independent of whatever else appears in the estimate. 

Whenever a norm, such as $L^2$, is used and no further comment is made then it is defined with respect to the fixed metric $\omega$ (rather than, say, the the sequence $k^{-1}\omega_k$ of projective approximations). 

There are various scale conventions used in the literature. In order to help the reader translate known results to fit with our scaling conventions, we state them all clearly here. (Our choice of conventions has been made to agree with those of Ma and Marinescu, whose asymptotic results we rely heavily on in \S\ref{hessian-asymps-section}.)
 
Given a positively curved Hermitian metric $h$ in $L$, we write $\omega = \frac{i}{2\pi} F_h \in c_1(L)$ for the associated Kähler form. We write $\Ric(\omega)$ for the Ricci form of $X$. Explicitly, $\Ric = -iF_K$ where $F_K$ is the curvature of the canonical bundle with Hermitian metric induced by $\omega$. We define the scalar curvature $S(\omega)$ by $S = 2(\Ric, \omega)$ where $(\cdot, \cdot)$ is the Riemannian inner-product associated to $\omega$. Equivalently, $S \omega^n = 2n \Ric \wedge \omega^{n-1}$. So, for example, for the Fubini--Study metric on $\mathcal O(1) \to \C\P^1$, $\omega$ has area 1, $\Ric = 4\pi \omega$ and $S = 8\pi$. 

With these conventions, if $e^{4\pi f t}h$ is a path of Hermitian metrics in $L$, then $\frac{\del S}{\del t} = D(f)$ for the fourth-order operator:
\begin{equation}
\label{linearisation-scalar-curvature}
D(f) = \Delta^2 f - 2(\Ric, 2i \delb\del f).
\end{equation}
Here $\Delta$ is the positive $\diff$-Laplacian. This is related to $\D^*\D$ by the equation
\begin{equation}
\label{principal-part}
2 \D^*\D(f) = D(f) + (\diff S, \diff f).
\end{equation}
The derivation of equations (\ref{linearisation-scalar-curvature}) and (\ref{principal-part}) is standard. See, for example, \cite{lebrun-simanca} and \cite{donaldson-fields}.

We mention also that sometimes in the literature the factor of $\frac{1}{4\pi}$ in the definition of $\mu$ is not used (see equation (\ref{mu}) above).

\subsection{A remark on uniformity with respect to the metric}
\label{uniformity_metric}

We pause to make a remark about the uniformity of the expansions in our Theorems \ref{asymptotics-hessians}, \ref{eigenvalue-theorem} and \ref{eigenvector-theorem} with respect to the Kähler metric. This follows in a straightforward fashion from the analogous uniformity in Theorems \ref{asymptotics-Bergman}, \ref{mm-first-expansion} and \ref{mm-second-expansion} together with the fact that all other estimates are metric-based. Consequently we state now once and for all that \emph{all results in this article are uniform with respect to the Kähler metric, provided it varies in a family of uniformly equivalent metrics which are compact for the $C^\infty$-topology}. We do not mention this point further in order to avoid littering the statements and proofs with repetitions of these facts.

\subsection{Acknowledgements}

I am indebted to Xiaonan Ma and George Marinescu who came to my aid with Theorems \ref{mm-first-expansion} and \ref{mm-second-expansion} when I needed the higher-order terms in these expansions. I would also like to thank Xiaonan Ma for kindly answering my many questions about such asymptotics. Thanks also go to Song Sun for interesting conversations concerning quantisation and Sam Lisi for explaining Hartman's work \cite{hartman} to me. Finally, I am grateful to the anonymous referees who commented on an earlier version of this article; they helped me greatly improve both the exposition and the strength of the results.

\section{Asymptotics of the Hessian of balancing energy}
\label{hessian-asymps-section}

In this section we prove Theorem \ref{asymptotics-hessians}. Our starting point is the following.
\begin{lemma}
\label{pointwise-identity}
Let $A, B \in i\u(m+1)$. The following identity holds pointwise on $\C\P^m$:
\[
4\pi H(A) H(B) + (\xi_A, \xi_B)_\mathrm{FS} = \tr(AB\mu)
\]
where $H(A)$ and $H(B)$ are defined in equation (\ref{H}), $\xi_A$ and $\xi_B$ are the holomorphic vector fields associated to $A$ and $B$ respectively, $(\cdot, \cdot)_\mathrm{FS}$ is the Hermitian Fubini--Study metric and $\mu$ is defined in equation (\ref{mu}).
\end{lemma}

This is Lemma 18 in \cite{fine} (to change to our current conventions, note that $H(A)$ and $H(B)$ must both be multiplied by $4\pi$ as must $(\cdot, \cdot)_\mathrm{FS}$ and $\mu$) and also appears in \cite{phong-sturm}.

\begin{proposition}
\label{hessian-balancing-identity}
Let $X^n \subset \C\P^m$ be a complex submanifold, $A, B \in i\u(m+1)$. The Hessian $\tr(A(P^*PB))$ of balancing energy at $X$ is given by:
\[
\re (\tr(AB \bar \mu) )
-
4\pi \int_X H(A)H(B) \frac{\omegaFS^n}{n!}
-
\int_X  (\diff H(A), \diff H(B) )_{\mathrm{FS}}\, \frac{\omegaFS^n}{n!}
\]
where $\omegaFS$ is the restriction of the Fubini--Study metric to $X$ and in the third term we first restrict $H(A),H(B)$ to $X$, then take the exterior derivative and finally compute the inner-product with the induced Riemannian metric on $X$.
\end{proposition}
\begin{proof}
At all points of $X$ we can split $\xi_A = \xi_A^\top + \xi_A^\perp$ into tangential and orthogonal components. Recall that  the Hessian of balancing energy is
\[
\tr(A(P^*PB))
= 
\int_X \re(\xi_A^\perp, \xi_B^\perp)_{\mathrm{FS}}\,\frac{\omegaFS^n}{n!}.
\]
Meanwhile, $\xi_A^\top = \nabla H(A)$ where we have taken the gradient of $H(A)$ over $X$ using the metric induced via Fubini--Study. It follows that 
\[
\re(\xi_A^\top, \xi_B^\top)_{\mathrm{FS}}
= 
\left(\diff H(A), \diff H(B)\right)_\mathrm{FS}
\]
where on the right-hand-side we first restrict to $X$, then take the exterior derivative and compute the inner-product with the induced Riemannian metric on $X$. From here the result follows by integrating Lemma \ref{pointwise-identity} over $X$.
\end{proof}

Recall that we fix a positively curved Hermitian metric $h$ in $L \to X$ and consider the sequence of embeddings $X \to \C\P^{n_k}$ provided by $L^2$-orthonormal bases of $H^0(X, L^k)$. We will prove Theorem \ref{asymptotics-hessians} by applying Proposition \ref{hessian-balancing-identity} to $A = Q_{f,k}$, $B= Q_{g,k}$ and computing the asymptotics of each of the terms. In fact, it will be clear in the course of our proof that this gives an asymptotic expansion with coefficients which are also symmetric bilinear forms in $f$ and $g$. (Strictly speaking, to verify this one needs to know $q_{f,g,1}$ for $f \neq g$ which is given in \cite{ma-marinescu}.) From this it follows that it suffices to prove Theorem \ref{asymptotics-hessians} in the case $f=g$. We write the details in this symmetric case, since it makes the notation far less cumbersome.

In the following, a subscript $k$ indicates that the object is defined via the $k^\mathrm{th}$ projective embedding $X \to \C\P^{n_k}$. So $\omega_k$ is the restriction of the Fubini--Study form to $X$, $|\cdot|_k$ is the associated norm on tensors etc. On the other hand, the absence of a subscript indicates that the object is defined with respect to the original choice of Kähler metric $\omega$.

\begin{lemma}
\label{FS-volume-form}
The volume form of the metric $\omega_k$ has the asymptotic expansion:
\[
\omega_k^n 
= 
k^n \omega^n
\left(1 - \frac{1}{32\pi^2 k^2} \Delta S + O(k^{-3}) \right).
\]
\end{lemma}
\begin{proof}
This follows from Theorem \ref{asymptotics-Bergman} which implies that 
\[
\omega_k 
=
k \omega - \frac{i}{2\pi} \delb\del
\left(
\frac{S(\omega)}{8\pi k} + O(k^{-2})
\right).
\]
\end{proof}

\begin{lemma}
\label{asymptotics-trQ2mu}
$\tr(Q_{f,k}^2 \bar\mu_k)$ has the following asymptotic expansion:
\[
4\pi k^{n+2}\int_X 
f^2
\,\frac{\omega^n}{n!}
+
k^{n+1}\int_X 
|\diff f|^2\,
\frac{\omega^n}{n!}
-\frac{1}{4\pi}k^n\int_X |\D f|^2\,\frac{\omega^n}{n!}
+
O(k^{n-1}).
\]
\end{lemma}

\begin{proof}
On restriction to $X \subset \C\P^m$, we can rewrite $\mu$ in terms of the Hermitian metric $h$ on $L$:
\[
\left(\mu|_X\right)|_{\alpha \beta}
 = 
\frac{(s_\alpha, s_\beta)}{4\pi \sum |s_\gamma|^2}
=
\frac{(s_\alpha, s_\beta)}{4\pi \rho_k},
\]
(where $(\cdot, \cdot)$ denotes the metric $h^k$ on $L^k$). It follows that 
\[
\tr (Q_{f,k}^2 \bar \mu_k)
=
\sum_{\alpha, \beta, \gamma} Q_{\alpha \beta} Q_{\beta \gamma} M_{\gamma \alpha},
\] 
where
\begin{eqnarray*}
Q_{\alpha \beta}
&=&
\int_X
\left(4\pi kf + \Delta f\right)(s_\alpha, s_\beta)\frac{\omega^n}{n!},
\\
M_{\gamma \alpha}
&=&
\int_X
\frac{(s_\gamma, s_\alpha)}{4\pi\rho_k}\frac{\omega_k^n}{n!}.
\end{eqnarray*}
By Lemma \ref{FS-volume-form} and the expansion of $\rho_k$ in Theorem \ref{asymptotics-Bergman},
\begin{eqnarray*}
\frac{\omega_k^n}{\rho_k}
&=&
\omega^n\left(
1 - \frac{S}{8\pi k} + \frac{S^2}{64\pi^2 k^2} - \frac{b_2}{k^2}
\right)
\left(
1 - \frac{1}{32\pi^2k^2}\Delta S
\right)
+
O(k^{-3}),\\
&=&
\omega^n\left(
1 - \frac{S}{8\pi k} 
+ 
\frac{S^2 - 64 \pi^2b_2 -2\Delta S}{64\pi^2 k^2} \right)
+ 
O(k^{-3}).
\end{eqnarray*}
Using this we have
\[
\tr (Q^2_{f,k} \bar\mu_k)
=
\frac{1}{4\pi}\int_X 
\left(
1 - \frac{S}{8\pi k} 
+ 
\frac{S^2 - 64 \pi^2b_2 -2\Delta S}{64\pi^2 k^2} + O(k^{-3})\right)
K_k\,
\frac{\omega^n}{n!}
\]
where 
\[
K_k 
= 
K_{4\pi kf + \Delta f, 4\pi kf + \Delta f,k}
\] 
is the restriction to the diagonal of the integral kernel of the composition of two Toeplitz operators, as discussed in \S\ref{asymptotic_results}. 

We now apply Theorem \ref{mm-second-expansion} to $K_k$. 
Taken literally, this result applies to $K_{f,g,k}$ for functions $f,g$ which are independent of $k$, whilst in our case the arguments depend linearly on $k$. However, since $K_{f,g,k}$ is in turn linear in $f$ and $g$ we can expand out and still use the asymptotic expansion. On doing this, the $k^{n+2}$-coefficient follows immediately. The $k^{n+1}$-coefficient is
\[
\int_X
\left(
2 f\Delta f + 4\pi q_{f,f,1} - \frac{1}{2} Sf^2
\right)\,\frac{\omega^n}{n!}
=
\int_X |\diff f|^2\,\frac{\omega^n}{n!},
\]
by virtue of Ma--Marinescu's formula for $q_{f,f,1}$. Finally, the $k^n$ coefficient is
\begin{multline*}
\int_X
\bigg(
\frac{1}{4\pi}(\Delta f)^2
+
q_{f,\Delta f,1} + q_{\Delta f, f,1}
+
4\pi q_{f,f,2}\\
-
\left(\frac{1}{4\pi}f \Delta f + \frac{1}{2} q_{f,f,1}\right)S
+
f^2\left(\frac{S^2 - 64\pi^2 b_2 - 2 \Delta S}{16\pi} \right)
\bigg)
\,\frac{\omega^n}{n!}.
\end{multline*}
Now we use the fact that $\int_X K_{f,g,k} = \int_X f K_{g,k}$. This tells us that $\int_X q_{f,g,j} = \int_X f q_{g,j}$. Applying the formulae in Theorems \ref{mm-first-expansion} and \ref{mm-second-expansion}, the above expression for the $k^n$ coefficient becomes
\[
\int_X
\left(
-\frac{1}{8\pi}(\Delta f)^2
+ 
\frac{1}{4\pi}f(\Ric, 2i\delb\del f)
+
\frac{1}{8\pi}Sf \Delta f
-
\frac{1}{8\pi} f^2 \Delta S
- 
\frac{1}{8\pi}S|\diff f|^2
\right)
\,\frac{\omega^n}{n!}.
\]
Next we use the following identities (which follow from integration by parts and Leibniz's rule)
\begin{equation}
\label{integration-parts}
\int_X f (\diff S, \diff f)\,\frac{\omega^n}{n!}
=
\frac{1}{2}
\int_X f^2 \Delta S\,\frac{\omega^n}{n!}
=
\int_X S \left(f \Delta f - |\diff f|^2\right)\,\frac{\omega^n}{n!}
\end{equation}
to write the $k^n$-coefficient as
\[
-\frac{1}{4\pi}\int_X |\D f|^2\,\frac{\omega^n}{n!},
\]
where we have used equation (\ref{principal-part}). 
\end{proof}

To compute the asymptotics of the two remaining terms in Proposition~\ref{hessian-balancing-identity} we first need an expansion for $H(Q_{f,k})$. 
\begin{lemma}
\label{asymptotics-HQ} 
There is an asymptotic expansion
\[
H(Q_{f,k})
=
kf - \frac{1}{32\pi^2k}D(f) + O(k^{-2}).
\]
where $D$ is the linearisation of the scalar curvature evaluated at $\omega$ (described in equation~(\ref{linearisation-scalar-curvature})). \end{lemma}

\begin{remark}\label{deriv_FSHilb}
The map $f \mapsto k^{-1}H(Q_{f,k})$ is the derivative of $\FS_k \circ \Hilb_k$ at $h$. Theorem \ref{asymptotics-Bergman} tells us that $\FS_k\circ \Hilb_k(h)$ converges to $h$, whilst this Lemma implies that the derivative converges uniformly to the identity on compact sets in $T_h\H$.
\end{remark}

\begin{proof}
First, a heuristic argument. Write $h_k$ for the metric on $L^k$ obtained by pulling back the Fubini--Study metric via an $L^2$-orthonormal basis of $H^0(X,L^k)$. We know that
\[
h_k
=
\frac{h^k}{\rho_k}
=
\left(
1 
- 
\frac{S}{8\pi k} 
+ 
\frac{S^2 - 64\pi^2 b_2}{64\pi^2k^2} 
+ 
O(k^{-3})
\right)k^{-n}h^k.
\label{hk-approximation}
\]
Now we differentiate this equation term-by-term with respect to $h$. In other words, we consider a path $e^{4\pi t f}h$ of Hermitian metrics, giving a path of Fubini--Study approximations and take the derivative of $h_k = h^k/\rho_k$ with respect to $t$. The infinitesimal change in the Fubini--Study metric is $4 \pi H(Q_{f,k})h_k$ whilst the change in $h^k$ is $4 \pi kf h^k$. Differentiating term-by-term on the right hand side of equation (\ref{hk-approximation}) (and assuming this is permitted!) we get
\[
H(Q_{f,k}) h_k
=
\left(
k f 
- 
\frac{1}{8\pi}S f 
+
\frac{S^2f -64 \pi^2 b_2 f - 2 D(f)}{64\pi^2 k}
+ 
O(k^{-2}) \right)k^{-n} h^k.
\]
Now, since $\frac{h^k}{h_k}=\rho_k$, we have the following expression for $H(Q{f,_k})$, in which we have suppressed all terms which ultimately contribute at $O(k^{-2})$.
\[
\left(
k f 
- 
\frac{1}{8\pi}S f 
+
\frac{S^2f -64 \pi^2 b_2 f - 2 D(f)}{64\pi^2 k}
\right) 
\left(1 + \frac{S}{8\pi k} + \frac{b_2}{k^2}\right)
\]
Multiplying this out gives the result.

For a rigorous proof, we write out $H(Q_{f,k})$ using $K_{f,k}$.
\begin{eqnarray*}
H(Q_{f,k})(x)
	&=&
	\sum_{\alpha, \beta}
	\int_X
	\left(4\pi k f + \Delta f\right)(y)
	(s_\alpha, s_\beta)(y)
	(s_\beta, s_\alpha)(x)
	\frac{1}{4\pi\rho_k(x)}\frac{\omega^n_y}{n!},\\
	&=&
	\frac{1}{4\pi \rho_k} K_{4\pi kf + \Delta f,k}.
\end{eqnarray*}
This means that $H(Q_{f,k})$ is the following product, in which we have suppressed terms of sufficiently high order that they ultimately contribute at $O(k^{-2})$.
\[
\left(
1 - \frac{S}{8\pi k} + \frac{S^2 - 64\pi^2b_2}{64\pi^2k^2}
\right)
\left(
kf 
+
\frac{Sf}{8\pi} 
+ 
k^{-1}\left(
f b_2 - \frac{\Delta^2 f }{32\pi^2}+ \frac{1}{8\pi^2}(\Ric, i \delb \del f) 
\right)
\right)
\]
(Again, we have considered the expansion of $K_{g,k}$ where $g$ depends on $k$, but since this dependence is linear as is the dependence of $K_{g,k}$ on $g$ this is permitted.) Multiplying these together and using equation (\ref{linearisation-scalar-curvature}) for  $D(f)$ completes the proof of the expansion.
\end{proof}

\begin{lemma}
\label{asymptotics-L2-norm-H}
There is the following expansion:
\[
\int_X H(Q_{k,f})^2 \frac{\omega_k^n}{n!}
=
k^{n+2} \int_X f^2\, \frac{\omega^n}{n!}
-
\frac{1}{8\pi^2}k^n \int_X |\D f|^2 \,\frac{\omega^n}{n!}
+O(k^{n-1}).
\]
\end{lemma}
\begin{proof}
It follows from Lemmas \ref{FS-volume-form} and \ref{asymptotics-HQ} that $\int_X H(Q_{f,k})^2\,\frac{\omega_k^n}{n!}$ is equal to
\[
\int_X 
k^n\left(k^2 f^2 - \frac{1}{16\pi^2}fD(f) +O(k^{-1})\right)
\left(1 - \frac{1}{32\pi^2k^2}\Delta S + O(k^{-3})\right)
\frac{\omega^n}{n!}.
\]
The coefficient of $k^{n+2}$ is $\int_Xf^2\,\frac{\omega^n}{n!}$ and that of $k^{n+1}$ vanishes. The $k^n$ coefficient is
\[
- \frac{1}{16\pi^2}
\int_X \left(
fD(f) + \frac{1}{2} f^2 \Delta S\right) \,\frac{\omega^n}{n!}
\]
This is equal to $-\frac{1}{8\pi^2}\int_X|\D f|^2 \,\frac{\omega^n}{n!}$  by equation (\ref{principal-part}).
\end{proof}

\begin{lemma}
\label{asymptotics-L2-norm-dH}
There is the following expansion:
\[
\int_X |\diff H(Q_{k,f})|^2_k\frac{\omega_k^n}{n!}
=
k^{n+1}\int_X |\diff f|^2\, \frac{\omega^n}{n!}
+
O(k^{n-1}).
\]
\end{lemma}

\begin{proof}
First note that for a fixed covector $\alpha$,
\begin{eqnarray*}
|\alpha|^2_k
&=&
|\alpha|^2_{k\omega + O(k^{-1})},\\
&=&
k^{-1}|\alpha|^2_{\omega + O(k^{-2})},\\
&=&
k^{-1}|\alpha|^2 + O(k^{-3}).
\end{eqnarray*}
From here and Lemmas \ref{FS-volume-form} and \ref{asymptotics-HQ} we see that the quantity we are computing is:
\[
\int_X
k^{n-1}
\left|k\diff f + O(k^{-1})\right|^2\left(1 +O(k^{-2}) \right)
\, \frac{\omega^n}{n!}.
\]
This gives the claimed expansion.
\end{proof}

\begin{proof}[Proof of Theorem \ref{asymptotics-hessians}]
We put together Proposition \ref{hessian-balancing-identity} with Lemmas \ref{asymptotics-trQ2mu}, \ref{asymptotics-L2-norm-H} and \ref{asymptotics-L2-norm-dH}. This gives that
\begin{eqnarray*}
\tr (P_k(Q_{f,k})^2)
&=&
\tr(Q_{f,k}^2\bar\mu_k)
-
4\pi\int_XH(Q_{k,f})^2 \,\frac{\omega_k^n}{n!}
- 
\int_X\left|\diff H(Q_{f,k})\right|_k^2 \,\frac{\omega_k^n}{n!},\\
&=&
\frac{k^n}{4\pi} \int_X |\D f|^2 \, \frac{\omega^n}{n!} + O(k^{n-1}).
\end{eqnarray*}

Finally we must check the claimed uniformity of the expansion in both $f$ and the Kähler metric. This follows ultimately from the analogous uniformity of Theorems \ref{asymptotics-Bergman}, \ref{mm-first-expansion} and \ref{mm-second-expansion}.
\end{proof}

\section{Strategy of the proof of Theorems \ref{eigenvalue-theorem} and \ref{eigenvector-theorem}}
\label{strategy}

\subsection{Overview of the induction}
 
We will prove the Theorems \ref{eigenvalue-theorem} and \ref{eigenvector-theorem} in tandem, by using induction. Recall our notation that $\lambda_j$ denotes the $j^\mathrm{th}$ eigenvalue of $\D^*\D$, whilst $\nu_{j,k}$ denotes the $j^\mathrm{th}$ eigenvalue of $P^*_kP_k$ (where the eigenvalues are repeated according to their multiplicity). We assume that $\Aut(X, L)/\C^*$ is discrete. We make the following inductive hypotheses.

\begin{inductive-hypotheses}
Let $r$ be a non-negative integer. We call the following statements the \emph{$r^\text{th}$ inductive hypotheses}.
\begin{enumerate}
\item[I1.]
For each $j = 0, \ldots , r$
\[
\nu_{j,k} = \frac{\lambda_j}{64\pi^3 k^2}  + O(k^{-3}).
\]
\item[I2.]
Write $F_{r,k} \subset i\u(n_k+1)$ for the span of the $\nu_{j,k}$-eigenspaces with $j \leq r$. Then there exists $C$ such that for all $A, B \in F_{r,k}$,
\[
\left| \tr(AB) - 16\pi^2k^n\langle H(A), H(B) \rangle_{L^2(\omega)} \right|
\leq 
Ck^{-1}\tr(A^2)^{1/2}\tr(B^2)^{1/2}.
\]
\item[I3.]
Fix integers $0<p< q \leq r$ such that 
\[
\lambda_{p-1} < \lambda_p = \lambda_{p+1} = \cdots = \lambda_q < \lambda_{q+1}.
\]
Write $V_{p}$ for the $\lambda_p$-eigenspace of $\D^*\D$ and write $F_{p,q,k}$ for the span of the $\nu_{j,k}$-eigenspaces of $P^*_kP_k$ with $p\leq j \leq q$.

Given $\phi \in V_{p}$, let $A_{\phi,k}$ denote the point of $F_{p,q,k}$ with $H(A_{\phi, k})$ nearest to $\phi$, as measured in $L^2$. Then
\[
\left\|
H(A_{\phi,k}) - \phi
\right\|^2_{L^2_2(\omega)}
=
O(k^{-1})
\]
and this estimate is uniform in $\phi$ if we require in addition that $\|\phi\|_{L^2} = 1$. 
\end{enumerate}
\end{inductive-hypotheses}

To prove Theorems \ref{eigenvalue-theorem} and \ref{eigenvector-theorem} we will induct on the \emph{spectral gaps} of $\D^*\D$. More precisely, suppose that $\lambda_r < \lambda_{r+1} = \cdots = \lambda_{s} < \lambda_{s+1}$. To carry out the induction we will prove that the $r^\text{th}$-inductive hypotheses implies the $s^\text{th}$-inductive hypotheses.

We next outline the main idea behind the proofs of I1--3. Let $E_r$ denote the span of the $\lambda_j$-eigenspaces for $j = 0, \ldots, r$. and write $F_{r,k}$ for the span of the $\nu_{j,k}$-eigenspaces for $j = 0, \ldots r$. Assuming that $\lambda_{r+1} > \lambda_r$ and $\nu_{r+1,k} > \nu_{r,k}$, the variational characterisation of eigenvalues gives that
\begin{eqnarray*}
\lambda_{r+1}
	&=&
		\min_{f \in E_r^\perp}
		\frac{\|\D f\|^2}{\|f\|^2},\\
\nu_{r+1,k} 
	&=& 
		\min_{B \in F_{r,k}^\perp}
		\frac{\| P(B) \|^2}{\tr (B^2)}.
\end{eqnarray*}
An important part of our proofs is to establish that if $f$ is orthogonal to $E_r$ then to highest order $Q_{f,k}$ is orthogonal to $F_{r,k}$. Similarly if $B$ is orthogonal to $F_{r,k}$ then to highest order $H(B)$ is orthogonal to $E_r$. In other words, under the derivatives of $\Hilb_k$ and $\FS_k$ the orthogonal complements asymptotically match up. Now, Theorem \ref{asymptotics-hessians} shows that $\|\D f \|$ can be used to control $\| P_k(Q_{f,k})\|$. Meanwhile, Proposition \ref{P*P-controls-D} below shows how $\| P_k(B)\|$ can be used to control $\|\D H(B) \|$. Since $E_r^\perp$ and $F_{r,k}^\perp$ are matching up, these results enable us to control $\nu_{r,k}$ in terms of~$\lambda_r$.

\subsection{The base of the induction}

\begin{lemma}\label{base}
Suppose that $\Aut(X,L)/\C^*$ is discrete. Then the $0^\text{th}$ inductive hypotheses are true.
\end{lemma}

\begin{proof}
Observe that $\ker \D^*\D = \ker \D$ is those functions whose Hamiltonian vector field is holomorphic. Such a vector field lifts to $L$ and so the assumption that $\Aut(X,L)/\C^*$ is discrete ensures that $\ker \D^*\D = \R$ and $\lambda_0 = 0$ is strictly smaller than $\lambda_1$. By definition, $P^*_kP_k$ vanishes on multiples of the identity, so $\nu_{0,k} = 0$ and so the first part of the base step is trivially satisfied.

Since $\mathrm{id}_k$ spans the $\nu_{0,k}$-eigenspace it suffices to prove the inequality in I2 just for this sequence of matrices. Now $H(\mathrm{id}_k) = \frac{1}{4\pi}$ is constant whilst $\tr(\mathrm{id}_k^2) = n_k+1 = Vk^n + O(k^{n-1})$ where $V= c_1(L)^n$ is the volume. So the left-hand-side is $O(k^{n-1})$ whilst the right-hand-side is $Vk^{n-1} + O(k^{n-2})$ and so the required constant $C$ can be found.

Finally, for I3, $H$ maps the $\nu_{0,k}$-eigenspace isomorphically onto the $\lambda_0$-eigen\-space, so there is no estimate required. 
\end{proof}

\section{A collection of estimates}
\label{estimates}

We now pause to collect various useful estimates. In reading the proofs of these estimates, remember that we habitually abuse notation in that $C$ denotes a constant which may change from line to line.

The first estimate in our collection is Lemma 15 in \cite{fine} (note that article does not use the factor of $\frac{1}{4\pi}$ in the definition of $\mu$). 

\begin{lemma}
$\| \bar\mu_k - \frac{1}{4\pi} \mathrm{id}_k\|_{\mathrm{op}} = O(k^{-1})$, where $\mathrm{id}_k \in i\u(n_k+1)$ is the identity matrix.
\end{lemma}

\begin{lemma}\label{bound-trABmu}
There is a constant $C$ such that for all $A, B \in i\u(n_k+1)$,
\[
\left| \tr(AB\bar\mu_k) - \frac{1}{4\pi} \tr(AB) \right |
\leq
C k^{-1} \tr(A^2)^{1/2}\tr(B^2)^{1/2}
\]
\end{lemma}
\begin{proof}
This follows from the bound $\|\bar\mu_k - \frac{1}{4\pi} \mathrm{id_k}\|_{\mathrm{op}} = O(k^{-1})$ and the standard fact that for Hermitian matrices $A, B,G$:
\[
\left|\tr(ABG)\right| \leq \tr(A^2)^{1/2}\tr(B^2)^{1/2}\|G\|_{\mathrm{op}}.
\]
(which can be proved, for example, by writing out the matrices in a basis for which $G$ is diagonal).
\end{proof}

\begin{lemma}
\label{L2k-L2-equivalence}
There exists a constant $C$ such that for any pair of functions $f, g, \in L^2$, 
\[
\left | 
k^{-n} \langle f, g \rangle_{L^2(\omega_k)} 
-
 \langle f, g \rangle_{L^2(\omega)}
\right|
\leq
C k^{-2} \| f\|_{L^2(\omega)}\| g\|_{L^2(\omega)}.
\]
Similarly for any pair of $L^2$-sections $s, t$ of $\,\Lambda^{0,1}\otimes TX$, 
\[
\left | 
k^{-n} \langle s, t \rangle_{L^2(\omega_k)} 
-
 \langle s, t \rangle_{L^2(\omega)}
\right|
\leq
C k^{-2} \| s\|_{L^2(\omega)}\| t\|_{L^2(\omega)}.
\]
\end{lemma}
\begin{proof} 
Since $\omega_k= k \omega + O(k^{-1})$ given any pair of functions $f,g$, we can write
\[
 \langle f, g \rangle_{L^2(\omega)}
=
k^n \langle f, g \rangle_{L^2(\omega_k)}
+
 \int_X f g \epsilon_k \, \frac{\omega^n}{n!}
\]
where $\epsilon_k =O(k^{n-2})$ in, say, $C^0$. By Cauchy--Schwarz,
\[
\left| \int_X fg \epsilon_k \, \frac{\omega^n}{n!} \right|
\leq
C k^{n-2} \| f \|_{L^2(\omega)} \| g\|_{L^2(\omega)}
\]
For some $C$ independent of $f, g$.

The proof for sections of $\Lambda^{0,1} \otimes TX$ is identical, once we have observed that scaling the metric by a constant does not change the metric on \hbox{$\Lambda^{0,1} \otimes TX$}.
\end{proof}

\begin{lemma}
\label{tr-controls-L2H}
There exists a constant $C$ such that for all $A \in i\u(n_k+1)$, 
\[
\| H(A) \|^2_{L^2(\omega)}
 \leq 
\left(\frac{1}{16\pi^2}k^{-n}+ Ck^{-n-1} \right) \tr(A^2).
\]
\end{lemma}
\begin{proof}
Lemma \ref{pointwise-identity} implies $4\pi H(A)^2 \leq \tr(A^2 \mu)$. Integrating over $X$ with respect to $\omega_k$ and applying Lemma \ref{bound-trABmu} with $B=A$ proves that
\[
\| H(A) \|^2_{L^2(\omega_k)} 
\leq 
\left(\frac{1}{16\pi^2}+ Ck^{-1} \right) \tr(A^2)
\]
Now it follows from Lemma \ref{L2k-L2-equivalence} that there is a constant $C$ such that for any function~$f$, 
\[
\| f \|^2_{L^2(\omega)} \leq \left(k^{-n} + Ck^{-n-2} \right)\| f \|^2_{L^2(\omega_k)}.
\]
Applying this to $H(A)$ completes the proof.
\end{proof}

\begin{lemma}
\label{Dk-D-control}
There exists a constant $C$ such that for all $f \in L^2_2$, 
\[
\left\| (k \D_k - \D)(f)\right \|_{L^2(\omega)} \leq Ck^{-2} \| f \|_{L^2_2(\omega)}
\]
where $\D_k$ is the operator defined by the metric $\omega_k$, whilst $\D$ is defined by $\omega$.
\end{lemma}
\begin{proof}
This follows from the fact that $\omega_k = k\omega + O(k^{-1})$. The Hamiltonian vector field generated by $f$ via $\omega$ is $\omega^{-1}(\diff f)$, where we write $\omega^{-1}$ for the inverse to contraction with $\omega$. Since $\D f= \delb(\omega^{-1}(\diff f))$ it follows that $\D_{k\omega} = k^{-1} \D_\omega$. Hence, 
\[
k\D_k = k\D_{k\omega + O(k^{-1})} =  \D_{\omega + O(k^{-2})}
\]

Now we appeal to the following fact: there exists constants $c$ and $C$ such that if $\omega, \omega'$ are two Kähler forms on $X$ with $\|\omega' - \omega\|_{C^1} \leq c$, then 
\[
\| (\D_{\omega'} - \D_{\omega}) f \|_{L^2}
\leq
C \| \omega' - \omega \|_{C^1} \| f\|_{L^2_2},
\]
where the norms are defined with respect to $\omega$. This is true because, provided $\omega'$ is sufficiently close to $\omega$, there is a constant $C$ such that for all $\omega'$ and all covectors~$\alpha$,
\[
\| \omega'^{-1}(\alpha)\|_{L^2_1} \leq C \| \alpha \|_{L^2_1}
\]
It now follows that
\begin{eqnarray*}
\| (\D_{\omega'} - \D_{\omega}) f\|_{L^2}
&\leq&
\left\| (\omega'^{-1} -  \omega^{-1})(\diff f) \right\|_{L^2_1},\\
&=&
\left\| \omega'^{-1}(\omega - \omega')\omega^{-1} (\diff f)\right\|_{L^2_1},\\
&\leq&
C \| \omega' - \omega \|_{C^1} \| f\|_{L^2_2}.
\end{eqnarray*}
\end{proof}

\begin{lemma}
\label{Lapk-Lap-control}
There exists a constant $C$ such that for all $f \in L^2_2$,
\[
\| (k \Delta_k - \Delta) f \|_{L^2} 
\leq 
C k^{-2} \| f \|_{L^2_2},
\]
where $\Delta_k$ denotes the $\omega_k$-Laplacian and the norms are defined with respect to $\omega$
\end{lemma}
\begin{proof}
This follows from the fact that $\omega_k = k\omega + O(k^{-1})$ together with the formula 
$
\Delta f\, \omega^n = 2 n i \delb \del f \wedge \omega^{n-1}
$
for the Laplacian on a Kähler manifold.
\end{proof}

\begin{lemma}
\label{L2k(Dk)-controls-L2(D)}
Assume that $\Aut(X,L)/\C^*$ is discrete. There exists a constant $C$ such that for all $f \in L^2_2$ and all large $k$,
\[
\| \D_k f\|^2_{L^2(\omega_k)}
\geq
\left(k^{n-2} - Ck^{n-4} \right) 
\| \D f \|^2_{L^2(\omega)}
-
Ck^{n-4}\| f\|^2_{L^2(\omega)}
\]
\end{lemma}
\begin{proof}
By Lemma \ref{Dk-D-control}, 
\[
\| k \D_k f \|_{L^2(\omega)} 
\geq
\| \D f\|_{L^2(\omega)} - Ck^{-2} \| f\|_{L^2_2(\omega)}.
\]
Since the symbol of $\D$ is injective, $\D$ satisfies an overdetermined-elliptic estimate: there is a constant $C$ such that for all $f \in L^2_2$,
\[
\| f \|_{L^2_2(\omega)} 
\leq 
C \big( \| f \|_{L^2(\omega)} + \| \D f\|_{L^2(\omega)} \big).
\]
It follows that
\[
\| k \D_k f\|_{L^2(\omega)}
\geq
\left(1 - Ck^{-2} \right) \| \D f \|_{L^2(\omega)}
- 
Ck^{-2} \| f\|_{L^2(\omega)}.
\]
Now $\Aut(X,L)/\C^*$ is discrete so $\ker \D = \ker \D_k$ is simply the constants. It follows that  unless $f$ is constant---in which case the required estimate is trivial---the right-hand-side is eventually positive for large $k$. Hence squaring gives  
\begin{eqnarray*}
\| \D_k f \|^2_{L^2(\omega)}
&\geq&
k^{-2}\left(1 - Ck^{-2} \right)^2 \| \D f\|^2_{L^2(\omega)}
-
Ck^{-4}\|\D f\|_{L^2(\omega)} \| f\|_{L^2(\omega)},\\
&\geq&
k^{-2}\left(1 - Ck^{-2} \right)^2 \| \D f\|^2_{L^2(\omega)}
-
Ck^{-4} \| f \|^2_{L^2(\omega)} - Ck^{-4}\| \D f \|^2_{L^2(\omega)},\\
&\geq&
\left(k^{-2}-Ck^{-4}\right) \|\D f\|^2_{L^2(\omega)}
-
Ck^{-4} \| f\|^2_{L^2(\omega)}.
\end{eqnarray*}

To complete the proof, we convert the left-hand-side to the $L^2(\omega_k)$-norm via Lemma \ref{L2k-L2-equivalence}. This gives a constant $C$ such that for any $L^2$-section $s$ of $\Lambda^{0,1}\otimes TX$, 
\[
\| s \|^2_{L^2(\omega_k)}
\geq
\left(k^n - C k^{n-2} \right) \| s \|^2_{L^2(\omega)}.
\]
Applying this to $s = \D_k f$ gives the result.
\end{proof}

\section{An asymptotic upper bound on $\nu_{j,k}$}
\label{upper-bound-section}

Our goal in this section is to prove the following.

\begin{proposition}
\label{upper-bound-result}
Suppose that $\lambda_r < \lambda_{r+1} = \cdots = \lambda_s< \lambda_{s+1}$ and, moreover, that the $r^\text{th}$-inductive hypotheses hold. Then for all  $j = r+1, \ldots, s$,
\[
\nu_{j,k} 
\leq 
\frac{\lambda_{j}}{64\pi^3k^2} + O(k^{-3}).
\]
\end{proposition}

\begin{proof}
We begin with the proof for $\nu_{r+1,k}$. Write $E_{r+1}$ for the sum of the eigen\-spaces of $\D^*\D$ with eigenvalue $\lambda \leq \lambda_{r+1}$. Similarly write $F_{r,k}$ for the sum of the eigen\-spaces of $P^*_kP_k$ with eigenvalue $\nu \leq \nu_{r,k}$. 

We know that $\dim E_{r+1} \geq r+2$ whilst $\dim F_{r,k} \geq r+1$ and that if $\dim F_{r,k} > r+1$ then $\nu_{r+1,k} = \nu_r$ which, by induction, is equal to $\frac{\lambda_r}{64\pi^3k^2} + O(k^{-3})$. So, for those values of $k$ for which $\dim F_{r,k} >r+1$ we certainly have $\nu_{r+1,k}$ satisfying the required upper bound; the effort is required to treat those values of $k$ for which $\dim F_{r,k}= r+1$.

For such values of $k$, we consider the map $\pi_k \colon E_{r+1} \to F_{r,k}$ where $\pi_k(f)$ is the orthogonal projection of $Q_{f,k}$ to $F_{r,k}$. Since $\dim E_{r+1} > \dim F_{r,k}$ it follows that $\pi_k$ has non-trivial kernel. For each such $k$, let $f_k \in \ker \pi_k$ have unit norm in $L^2$. Since $Q_{f_k,k} \in  F_{r,k}^\perp$, we have that
\[
\nu_{r+1,k} \leq \frac{\| P_k Q_{f_k,k} \|^2}{\tr(Q_{f_k,k}^2)}.
\]

Now, to bound the numerator above we apply Theorem \ref{asymptotics-hessians}, which applies uniformly to all $f$ in the unit sphere of $E_{r+1}$. This gives, for all $l$ and for any integer $k >0$ for which $\dim F_{r,k} =r+1$, 
\[
\| P_l Q_{f_k,l} \|^2 \leq \frac{l^n}{4\pi} \lambda_{r+1} + Cl^{n-1}
\]
for some $C$ (since $f_k$ is an eigenvector of $\D^*\D$ with eigenvalue $\lambda \leq \lambda_{r+1}$). Putting $l=k$ we get
\begin{equation}\label{numerator_bound}
\| P_k Q_{f_k,k} \|^2 \leq \frac{k^n}{4\pi} \lambda_{r+1} + Ck^{n-1}.
\end{equation}

Meanwhile, by definition (\ref{Q-matrix}) of $Q_{f,k}$, $\tr (Q_{f_k,k}^2)$ is given by the double integral
\[
\int_{X \times X}
\left(4\pi kf_k + \Delta f_k\right)(x)
\left(4\pi kf_k + \Delta f_k\right)(y)
(s_\alpha,s_\beta)(x)(s_\beta, s_\alpha)(y)
\frac{\omega_x^n\wedge\omega_y^n}{(n!)^2}
\]
We can rewrite this using Toeplitz integral kernels as
\[
\int_X\left(
16\pi^2k^2  f_k K_{f_k,k} 
+
4\pi k f_k K_{\Delta f_k, k}
+
4\pi k \Delta f_k K_{f_k,k}
+
\Delta f_k K_{\Delta f_k,k} 
\right)\,\frac{\omega^n}{n!}.
\]
Now, since the $f_k$ lie in the unit sphere of $E_{r+1}$ which is compact in the $C^\infty$-topology, Theorem \ref{mm-first-expansion} gives that $K_{f_k,k} = k^n f_k + O(k^{n-1})$ and $K_{\Delta f_k, k} = O(k^n)$. (Again we use the uniformity of Theorem \ref{mm-first-expansion} in $f$ together with a diagonal argument to deduce these estimates, just as Theorem \ref{asymptotics-hessians} was applied to derive (\ref{numerator_bound}) above.) From here and the fact that $\| f_k\|_{L^2} = 1$, it follows that
\begin{equation}\label{denominator_bound}
\tr(Q_{f_k,k}^2)
=
16\pi^2k^{n+2} + O(k^{n+1}).
\end{equation}
The bounds (\ref{numerator_bound}) and (\ref{denominator_bound}) combine to give
\[
\nu_{r+1,k} \leq \frac{\lambda_{r+1}}{64\pi^3k^2} +O(k^{-3})
\]
as required.

Suppose now that $\lambda_{r+2} = \lambda_{r+1}$. Then $\dim E \geq r+3$ and we can repeat the argument, this time projecting onto the span $F_{r+1,k}$ of eigenspaces with $\nu \leq \nu_{r+1,k}$. This time $\dim F_{r+1,k} \geq r+2$ and is strictly greater only if $\nu_{r+2,k} = \nu_{r+1,k}$, which we have just bounded above. Whenever $\dim F_{r+1,k} = r+2$ we have an element $f_k \in E_{r+1}$ with $Q_{f_k,k} \in F_{r+1,k}^\perp$ and this gives a sequence enabling us to bound $\nu_{r+2,k}$ above. Continuing in this way until we reach $\lambda_s$ completes the proof.
\end{proof}

\section{An asymptotic lower bound on $\nu_{j,k}$}
\label{lower-bound-section}

The goal of this section is to prove the following.

\begin{proposition}\label{lower-bound-result}
Assume that $\Aut(X,L)/\C^*$ is discrete. Suppose $\lambda_r < \lambda_{r+1}$ and that the $r^{\text{th}}$-inductive hypotheses hold. Then
\[
\nu_{r+1,k} \geq \frac{\lambda_r}{64\pi^3 k^2} + O(k^{-3}).
\]
\end{proposition}

The proof is somewhat lengthy and will finally be established in \S\ref{lower-bound-proof-section}. As is explained in \S\ref{strategy}, a key step is to use $\| P_k(A) \|$ to control $\| \D H(A) \|_{L^2}$. This is the content of Proposition \ref{P*P-controls-D} in \S\ref{2ff-asymptotics} below. 

The result and its proof are related to work of Phong and Sturm \cite{phong-sturm} concerning embeddings with $R$-bounded geometry, which we now briefly describe.

\begin{definition}\label{R-bounded}
Fix $R >0$ and an integer $s$ greater than 4. Given $b \in \B_k$ we write $\omega_b \in kc_1(L)$ for the corresponding Kähler metric. We say \emph{$b$ has $R$-bounded geometry in $C^s$} if $\omega_b > R^{-1}k\omega$ and
\[
\| \omega_b - k\omega \|_{C^s(k\omega)} < R.
\]
\end{definition}

It follows from Theorem \ref{asymptotics-Bergman} that the sequence $\Hilb_k(h)$ that we are interested in has $R$-bounded geometry for  all large $k$ and any choice of $s$. 

\begin{theorem}[Phong--Sturm, \cite{phong-sturm}]
\label{phong-sturm-theorem}
Assume that $\Aut(X, L)/\C^*$ is discrete. Fix $R>0$ and an integer $s \geq 4$. There exists a constant $C$ such that for all points $b \in \B_k$ with $R$-bounded geometry in $C^s$,
\[
\nu_{1,k}(b) \geq C k^{-2}
\]
where $\nu_{1,k}(b)$ denotes the first non-zero eigenvalue of the Hessian of balancing energy at $b$.
\end{theorem}

This result is in turn an improvement on the lower bound $ \nu_{1,k}(b) \geq Ck^{-4}$, due to Donaldson \cite{donaldson-1}. Phong and Sturm's use of the second fundamental form of $X \to \C\P^{n_k}$ in their proof of Theorem \ref{phong-sturm-theorem} was what led us to the geometric relation between the two Hessians $\D^*\D$ and $P^*P$ described in Proposition~\ref{two-hessians-equation} below. 

\subsection{The Hessians for a complex projective submanifold}\label{hessians-complex-submfd}

We begin by recalling some standard material on the second fundamental form. In part, we follow the treatment in \cite{griffiths-harris}. Let $E \to Y$ be a holomorphic vector bundle over a complex manifold; suppose that $S \to E$ is a holomorphic sub-bundle, with quotient $Q$. Suppose moreover that we have a Hermitian metric in $E$. We write $\nabla^E$ for the corresponding Chern connection. By restriction, we obtain a Hermitian metric on $S$. By identifying $Q$ with the orthogonal complement of $S$ we also obtain a Hermitian metric on $Q$. These metrics give us Chern connections $\nabla^S$ and $\nabla^Q$ in $S$ and $Q$ respectively.

It is straightforward to check that $\nabla^S$ is the composition of $\nabla^E$ followed by projection to $S$. Meanwhile the composition of $\nabla^E$ with projection to $Q$ gives the second fundamental form $F$ of $S \to E$:
\[
C^\infty(S) \stackrel{\nabla^E}{\to} \Omega^1(E) \to \Omega^1(Q).
\]
$F$ is tensorial, being given by a $1$-form with values in $\Hom(S,Q)$. Moreover, since $S$ is a holomorphic sub-bundle, the $(0,1)$-component of $\nabla^E$ leaves $S$ invariant; hence $F$ is a section of  $\Lambda^{1,0} \otimes \Hom (S, Q)$.

$F$ measures the failure of $S$ to be a parallel sub-bundle of $E$. We can also approach the second fundamental form by considering the failure of $S^\perp$, the orthogonal complement of $S$, to be a \emph{holomorphic} sub-bundle of $E$. To do this, we consider the composition of $\delb_E$ with orthogonal projection to $S$ to obtain a map:
\[
C^\infty(S^\perp) \stackrel{\delb_E}{\to}
\Omega^{0,1}(E) \to \Omega^{0,1}(S).
\]
Again, this is tensorial, given by a section of $\Lambda^{0,1}\otimes \Hom(S^\perp, S)$. Moreover, if we identify $Q \cong S^\perp$ in the obvious way, then this section is $F^*$, i.e., the $(0,1)$-form with values in $\Hom (Q,S)$ obtained by conjugating the $(1,0)$-form factor of $F$ and taking the adjoint of the $\Hom(S,Q)$ factor. 

We will exploit a standard calculation which relates these objects to the curvatures of $E$, $S$ and $Q$. We write $F \wedge F^* \in \Lambda^{1,1} \otimes \End(Q)$ for the wedge product on forms tensored with the composition on endomorphisms. Similarly, $F^* \wedge F \in \Lambda^{1,1} \otimes \End(S)$. In the following equations, $R(S)$, $R(Q)$ and $R(E)$ denote the curvatures of the Chern connections in $S$, $Q$ and $E$ respectively. The $C^\infty$-splitting $E = S \oplus Q$ induces a splitting
\[
\End(E)
=
\End(S) \oplus \Hom(S, Q) \oplus \Hom (Q,S) \oplus \End(Q).
\]
We write $R(E)|_S$ and $R(E)|_Q$ for the components of $R(E)$ in $\End(S)$ and $\End(Q)$ respectively. The following is proved, for example, on page 78 of \cite{griffiths-harris}.
\begin{eqnarray}
F^* \wedge F
	&=&
		R(S) - R(E)|_S,
		\label{2ff-1}\\
F\wedge F^*
	&=& R(Q) - R(E)|_Q
		\label{2ff-2}
\end{eqnarray}

Next we assume that $Y$ carries a Hermitian metric. This gives an identification $\Lambda^{1,0} \cong (\Lambda^{0,1})^*$. Using this, we can interpret $F$ as a homomorphism 
\[
F \colon \Lambda^{0,1}\otimes S \to Q.
\] 
We can also regard $F^*$ as a homomorphism 
\[
F^* \colon Q \to \Lambda^{0,1} \otimes S.
\]
The fact we need is that these two maps are adjoints with respect to the fibrewise Hermitian metrics on $\Lambda^{0,1}\otimes S$ and $Q$, which can easily be verified directly from the definitions.

In what follows, we will be interested in the two following compositions:
\begin{equation}
\label{F*F}
\xymatrix{
\Lambda^{0,1} \otimes S
\ar[r]^-{F}
&
Q
\ar[r]^-{F^*}
&
\Lambda^{0,1} \otimes S
}
\end{equation}
\begin{equation}
\label{FF*}
\xymatrix{
Q
\ar[r]^-{F^*}
&
\Lambda^{0,1} \otimes S
\ar[r]^-{F}
&
Q
}
\end{equation}
Via $\Lambda^{1,0} \cong (\Lambda^{0,1})^*$, we may think of (\ref{F*F}) as defining an element of $\Lambda^{1,1}\otimes \End(S)$. If we choose an orthonormal basis $e_1, \ldots, e_n$ for $\Lambda^{1,0}$, we can write $F = \sum e_j \otimes F_j$ for $F_j \in \Hom(S, Q)$. Then $F^* = \sum \bar e_j \otimes F_j^*$ and we have
\begin{eqnarray*}
F^*F &=& \sum_{i,j} e_i \wedge \bar e_j \otimes F_j^*F_i,\\
FF^* &=& \sum_i F_i F_i^*.
\end{eqnarray*} 
In other words, under the identification $\End(\Lambda^{0,1} \otimes S) \cong \Lambda^{1,1}\otimes \End(S)$, $F^*F $ is identified with $- F^* \wedge F$. (To see that the sign here is correct, observe that $F^*F$ is identified with a \emph{positive} $(1,1)$-form with values in $\End(S)$, on the other hand, $F^*\wedge F$ is negative; see the discussion on pages 78--79 of \cite{griffiths-harris}.) Meanwhile, $FF^* = \tr_Y(F\wedge F^*)$, where we have taken the trace of the $\Lambda^{1,1}$-component of $F\wedge F^*$ via the Hermitian metric on $Y$.

After this somewhat lengthy digression, we now return to the case of a complex submanifold $Y \subset \C\P^m$. We apply the above discussion to the bundles $E= T\C\P^m|_Y$, $S = TY$ and the normal bundle $N$ as the quotient. The Hermitian metrics are all induced by the Fubini--Study metric on $\C\P^m$. Throughout we freely identify $N$ with the orthogonal complement of $TY$.

Recall that in this situation, we have two linear maps:
\[
P \colon i \u(m+1) \to C^\infty(N),\quad\quad
P(A) = \xi_A^\bot.
\]
\[
H \colon i\u(m+1) \to C^\infty(Y,\R), \quad\quad
H(A) = \tr(A\mu)|_Y
\]
Using the Killing form on $i\u(m+1)$ and the $L^2$-inner-product defined by the Fubini--Study metric on both ranges, we can define the adjoints $P^*$ and $H^*$.

\begin{proposition}
\label{two-hessians-equation}
Let $Y \subset \C\P^m$ be a complex submanifold. Then
\[
H^*\D^*\D H = P^* F F^* P
\] 
where $\D$ is defined by the restriction of the Fubini--Study metric to $Y$ and $F F^*$ is the endomorphism of $N$ built from the second fundamental form of $X$, as in (\ref{FF*}).
\end{proposition}

\begin{proof}
Let $A \in i\u(m+1)$. By definition, $\D \left(H(A)\right) = J \delb \left(\nabla H(A)\right) = J \delb (\xi_A^\top)$. But $\xi_A$ is a holomorphic section of $T\C\P^m$ and so $\delb (\xi_A^\top) = - \delb (\xi_A^\perp)$. Now $TY$ is a holomorphic sub-bundle of $T\C\P^N$, so $\delb (\xi_A^\perp) $ already takes values in $TY$. It follows that $\delb(\xi_A^\perp) = F^*(\xi_A^\perp)$ and hence $\D H(A) = - J F^*(P(A))$. Since $J^* = J^{-1}$ the result now follows.
\end{proof}

The proof may just be a matter of unwinding the definitions, but this result is absolutely crucial to what follows, since it gives the geometric relation between the two Hessians.

\subsection{Asymptotics of the second fundamental form}
\label{2ff-asymptotics}

We now turn to the asymptotic behaviour of the above picture. Recall that we have a positive Hermitian metric $h$ in $L \to X$ which gives embeddings $X \to \C\P^{n_k}$ via $L^2$-orthonormal bases of $H^0(X, L^k)$. Applying the discussion in \S\ref{hessians-complex-submfd} to these embeddings, we obtain objects $P_k, F_k, \D_k$ etc., where the subscript denotes the dependence on $k$. 

Our goal is to understand the asymptotic behaviour of $F_kF_k^*$, but it proves easier to go first via $F_k^*F_k \in \End(\Lambda^{0,1} \otimes TX)$. We will prove that asymptotically $F_k^*F_k$ is $4\pi S_k$ for a certain self-adjoint projection $S_k$. To define $S_k$, we use the metric  $\omega_k$ to identify:
\[
\iota_k \colon \Lambda^{0,1}\otimes TX \to TX \otimes TX.
\]
Write $\pi \colon TX \otimes TX \to TX \otimes TX$ for the projection onto the symmetric tensors; for $\alpha, \beta \in T^*X$ and $T \in TX\otimes TX$, 
\[
\pi(T)(\alpha, \beta) = \frac{1}{2} T(\alpha, \beta) + \frac{1}{2}T(\beta, \alpha).
\]
We then set $S_k = \iota_k^{-1}\circ \pi \circ \iota_k$.
 
\begin{lemma}
$\| F^*_kF_k - 4\pi S_k \|_{C^0(\mathrm{op})} = O(k^{-1}).$ 
Here we use the $C^0$-norm on sections of~$\,\End(\Lambda^{0,1} \otimes TX)$ associated to the fibrewise operator norm.
\end{lemma}
\begin{proof}
Recall from \S\ref{hessians-complex-submfd} that under the identification of $\End(\Lambda^{0,1} \otimes TX)$ with $\Lambda^{1,1} \otimes \End(TX)$, $F^*_kF_k$ is identified with $- F^*_k \wedge F_k$. Now, by (\ref{2ff-1}), we have:
\[
F^*_k \wedge F_k =  R(TX, \omega_k) - R(T\C\P^m)|_{TX}.
\]
With our scaling conventions, the curvature tensor of $\C\P^m$ is given by
\[
R(u_1, \bar{u}_2, u_3, \bar{u}_4)
=
2\pi \left(
g(u_1, \bar u_2) g(u_3, \bar u_4) + g(u_1, \bar u_4) g(u_3, \bar u_2)
\right)
\]
where $g$ is the Fubini--Study Hermitian inner-product. In index notation, this reads
\[
R_{i \bar \jmath k \bar l}
=
2\pi\left(
g_{i \bar \jmath} g_{k \bar l} 
+ 
g_{i \bar l} g_{k \bar \jmath}
\right)
\]
Raising indices, to interpret this as a map $T\C\P^{n_k} \otimes T\C\P^{n_k} \to T\C\P^{n_k} \otimes T\C\P^{n_k}$, we obtain
\[
R_{i \phantom{j} k \phantom{l}}^{\phantom{i} j \phantom{k} l}
=
2\pi 
\left(
\delta_{i \phantom{j}}^{\phantom{i} j}
\delta_{k \phantom{l}}^{\phantom{k} l}
+
\delta_{i \phantom{l}}^{\phantom{i} l}
\delta_{k \phantom{j}}^{\phantom{k} j}
\right)
\]
In other words, $R(T\C\P^{n_k})|_{TX}$ corresponds to $4\pi S_k$.

It remains to show that, when we interpret it via $\omega_k$ as an endomorphism of $TX \otimes TX$, we have $\| R(TX, \omega_k)\|_{C^0(\mathrm{op})} = O(k^{-1})$. This follows from the fact that $\omega_k = k\omega + O(k^{-1})$. To see this, write $R_k \in \End(TX \otimes TX))$ for the curvature tensor $R(TX,\omega_k)$ with the index appropriately raised via $\omega_k$. Now, if $\omega_k$ were \emph{exactly} equal to $k\omega$, the curvature tensor thought of as an element of $\Lambda^{1,1}\otimes \End(TX)$ would be independent of $k$; raising the index with $\omega_k$ would introduce a factor of $k^{-1}$ and we would have $R_k = k^{-1}R_1$. The lower order terms in $\omega_k = k\omega + O(k^{-1})$ only affect this argument at $O(k^{-1})$.
\end{proof}

Let $T_k \in \End(N)$ denote the orthogonal projection in $N$ onto the image of $F_k \colon \Lambda^{0,1}\otimes TX \to N$.

\begin{lemma}
\label{control-FF*}
$\| F_k F^*_k - 4\pi T_k \|_{C^0(\mathrm{op})} = O(k^{-1})$. Here we use the $C^0$-norm on sections of $\End(N)$ associated to the fibrewise operator norm.
\end{lemma}
\begin{proof}
Clearly $\ker F_kF^*_k = \ker T_k$. Since $F_kF^*_k$ is self-adjoint, it remains to show that all of its non-zero eigenvalues are $4\pi + O(k^{-1})$. But the non-zero eigenvalues of $F_kF_k^*$ and $F^*_kF_k$ are identical (the eigenvectors are matched up by $F^*_k$) hence the result follows from the previous Lemma.
\end{proof}

The conclusion of this discussion is that $\|P_k(A)\|$ together with the $L^2$-norm of $H(A)$ controls the $L^2$-norm of $\D H(A)$:

\begin{proposition}
\label{P*P-controls-D}
There is a constant $C$ such that for all $A \in i\u(n_k+1)$,
\[
\| \D H(A) \|^2_{L^2(\omega)}
\leq
\left(4\pi k^{2-n} + C k^{1-n} \right) 
\| P_k(A) \|^2
+
C k^{-2} \| H(A)\|^2_{L^2(\omega)}.
\]
\end{proposition}
\begin{proof}
By Proposition \ref{two-hessians-equation} and Lemma \ref{control-FF*}:
\begin{eqnarray*}
\| \D_k H(A)\|^2_{L^2(\omega_k)}
	&=&
		\langle F_kF^*_k P_k(A), P_k(A) \rangle_{L^2(\omega_k)}\\
	&=&
		\left\langle \left(4\pi T_k + O(k^{-1})\right)P_k(A),
			 P_k(A) \right\rangle_{L^2(\omega_k)}\\
	&\leq&
		\left(4\pi + Ck^{-1} \right)
			\|P_kA\|^2,
\end{eqnarray*}
where in the second line $T_k$ is the projection operator of Lemma \ref{control-FF*}. The result now follows from Lemma \ref{L2k(Dk)-controls-L2(D)} which shows how both $\| \D_k f\|^2_{L^2(\omega_k)}$ and $\| f\|^2_{L^2(\omega)}$ control $\| \D f \|^2_{L^2(\omega)}$.
\end{proof}

\subsection{Completing the proof of the lower bound}
\label{lower-bound-proof-section}

We will now move to finish the proof of Proposition \ref{lower-bound-result}. We recall our assumptions: that $\lambda_r < \lambda_{r+1}$ and that the $r^{\text{th}}$-inductive hypotheses hold. Recall also that $E_r$ denotes the span of the $\D^*\D$-eigenspaces with eigenvalue $\lambda \leq \lambda_r$ whilst $F_{r,k}$ denotes the span of the $P^*_kP_k$-eigenspaces with eigenvalue $\nu \leq \nu_{r,k}$. Since $\lambda_{r+1} > \lambda_r$ we have that $\dim E_{r} = r+1$. We fix an orthonormal basis $\phi_0, \ldots , \phi_r$ of $E_{r}$ in which $\phi_j$ is a $\lambda_j$-eigenvector of $\D^*\D$. Now we let $A_{j,k}$ denote the points of $i\u(n_k+1)$  furnished by the inductive hypothesis I3, for which $H(A_{j,k})$ approaches $\phi_j$. More precisely, let $p<q\leq r$ be such that 
\[
\lambda_{p-1} < \lambda_p = \cdots = \lambda_q < \lambda_{q+1}
\]
Write $F_{p,q,k} \subset i\u(n_k+1)$ for the span of the $\nu_{j,k}$-eigenspaces of $P^*_kP_k$ with $p\leq j\leq q$. Then, for $p\leq j \leq q$ we set $A_{j,k}$ to be the point of $F_{p,q,k}$ with $H(A_{j,k})$ closest in $L^2$ to $\phi_j$. 

\begin{lemma}\label{variation_nu_r+1}
Let $W_k \subset F_{r,k}$ denote the span of the $A_{j,k}$. Then
\[
\nu_{r+1,k}
\geq
\min_{B \in W_k^\perp}
\frac{\|P_kB\|^2}{\tr (B^2)}
\]
\end{lemma}
\begin{proof}
We begin by proving that $\dim W_k = r+1$. To see this we apply the inductive hypothesis I2, which says that there is a constant $C$ such that
\[
\left| \tr(A_{i,k}A_{j,k}) - 16\pi^2k^n\langle H(A_{i,k}), H(A_{j,k}) \rangle_{L^2(\omega)}\right|
\leq C k^{-1} \tr(A^2_{i,k})^{1/2}\tr(A^2_{j,k})^{1/2}.
\]
Now, $H(A_{j,k}) = \phi_j + O(k^{-1/2})$ in $L^2_2(\omega)$ and so in particular in $L^2(\omega)$. Moreover, the $\phi_j$ are an orthonormal basis for $E_r$; this gives 
\begin{eqnarray*}
\tr(A_{i,k}^2) &=& 16\pi^2k^n + O(k^{n-1/2}),\\
\tr(A_{i,k}A_{j,k}) &=& O(k^{n-1/2}) \ \text{for }i \neq j.
\end{eqnarray*}
Since the inner-products of pairwise distinct $A_{i,k}$ are of strictly lower order than the lengths of the vectors themselves it follows that they are linearly independent.

So $W_k$ is an $(r+1)$-dimensional $P^*_kP_k$-invariant subspace. It follows that the minimal eigenvalue of $P^*_kP_k$ on the orthogonal complement $W_k^\perp$ is at most the $(r+2)^\text{th}$-eigenvalue, namely $\nu_{r+1,k}$. The result now follows from the variational characterisation of eigenvalues.
\end{proof}

We remark in passing that once we have proved Proposition \ref{lower-bound-result} and hence $\nu_{r+1,k} > \nu_{r,k}$ it will follow that $\dim F_{r,k} = r+1$; hence $W_k = F_{r,k}$ and the inequality in Lemma \ref{variation_nu_r+1} will be seen to be an equality.

The next important step in the proof is to show that if $B \in W_{k}^\perp$ then to highest order $H(B)$ lies in $E_{r}^\perp$. This is the content of the following result.

\begin{proposition}
\label{nu-espaces-converging}
Let $\psi \in L^2_2$ and let $M_k \in i\u(n_k+1)$ be a sequence of $P^*_kP_k$-eigenvectors satisfying the two conditions:
\begin{enumerate}
\item
$\tr(M_k^2) = 16\pi^2k^n + O(k^{n-1/2})$.
\item
$\|H(M_k) - \psi \|^2_{L^2_2} = O(k^{-1})$.
\end{enumerate}
Then there is a constant $C$ such that for all $B\in i\u(n_k+1)$ with $\tr(BM_k) = 0$, 
\[
\left| \langle H(B), \psi \rangle_{L^2(\omega)} \right|^2 \leq C k^{-n-1} \tr(B^2).
\]
\end{proposition}

\begin{proof}
Since $\tr(BM_k) =0$ and $M_{k}$ is an eigenvector of $P^*_kP_k$ we have, 
\[
\tr\Big(B (P^*_kP_k M_k) \Big) = 0.
\]
Hence it follows from Proposition \ref{hessian-balancing-identity} that
\begin{equation}
\label{hessian-on-evec}
4\pi \langle H(B), H(M_k) \rangle_{L^2(\omega_k)}
+
\langle \diff H(B), \diff H(M_k) \rangle_{L^2(\omega_k)}
=
\re \left(\tr(M_kB \bar \mu_k)\right)
\end{equation}
To control the right-hand-side we apply Lemma \ref{bound-trABmu}, together with $\tr(M_kB) =0$ and $\tr(M^2_k) = 16\pi^2k^{n} + O(k^{n-1/2})$ to deduce that
\[
\left| \tr(M_kB\bar\mu_k) \right|^2
\leq 
C k^{n-2} \tr(B^2)
\]
for some $C$ independent of $B$. To ease the notation, we write 
\[
\Theta_k
=
H(M_k) + \frac{1}{4\pi} \Delta_k H(M_k)
\]
where $\Delta_k$ is the $\omega_k$-Laplacian. It follows  from (\ref{hessian-on-evec}) and the bound on $\tr(M_kB \bar\mu_k)$ that
\[
\left| \left\langle H(B), \Theta_k \right\rangle_{L^2(\omega_k)}\right|^2
\leq
Ck^{n-2} \tr(B^2).
\]
We now pass from the $L^2(\omega_k)$-norm to the $L^2(\omega)$-norm via Lemma \ref{L2k-L2-equivalence}. This gives a constant $C$ such that for all $B$ with $\tr(M_kB) = 0$,
\begin{equation}
\label{intermediate-ineq}
\left| \langle H(B), \Theta_k \rangle_{L^2(\omega)} \right|^2
\leq
Ck^{-n-2} \tr(B^2)
+
Ck^{-4} \| H(B)\|^2_{L^2(\omega)} \| \Theta_k \|^2_{L^2(\omega)}.
\end{equation}

The next step is to show that $\Theta_k = \psi+ O(k^{-1/2})$ in $L^2(\omega)$. Since $H(M_k) = \psi + O(k^{-1/2})$ in $L^2_2(\omega)$, it suffices to prove $\Delta_k H(M_k) = O(k^{-1})$ in $L^2(\omega)$. By Lemma~\ref{Lapk-Lap-control}, there is a constant $C$ such that
\begin{eqnarray*}
k \| \Delta_k H(M_k) \|_{L^2(\omega)}
&\leq&
\| \Delta H(M_k) \|_{L^2(\omega)} 
+ 
Ck^{-2}\|H(M_k)\|_{L^2_2(\omega)}\\
&\leq&
C \| H(M_k) \|_{L^2_2(\omega)}
\end{eqnarray*}
Since $H(M_k) = O(1)$ in $L^2_2(\omega)$ it follows that $\Delta_k H(M_k) = O(k^{-1})$ in $L^2(\omega)$ and hence that $\Theta_k = \psi + O(k^{-1/2})$ in $L^2(\omega)$.

We can now complete the proof. From (\ref{intermediate-ineq}) we deduce that
\[
\left|\langle H(B), \psi  \rangle_{L^2(\omega)}\right|^2
\leq
Ck^{-n-2} \tr(B^2)
+
Ck^{-1} \| H(B) \|^2_{L^2(\omega)}.
\]
Lemma \ref{tr-controls-L2H} gives a constant $C$ such that $\| H(B) \|^2_{L^2(\omega)} \leq Ck^{-n} \tr(B^2)$. This finishes the proof.
\end{proof}

With this result in hand, we now finish the proof of Proposition \ref{lower-bound-result}: assuming $\Aut(X,L)/\C^*$ is discrete, that $\lambda_{r+1}>\lambda_r$ together with the $r^\text{th}$ inductive hypotheses, we will prove the lower bound
\[
\nu_{r+1,k} \geq \frac{\lambda_{r+1}}{64\pi^3 k^2} + O(k^{-3}).
\]

\begin{proof}[Proof of Proposition \ref{lower-bound-result}]
By Lemma \ref{variation_nu_r+1}, we know that 
\[
\nu_{r+1,k} 
\geq
\min_{B\in W_k^\perp} \frac{\| P_kB\|^2}{\tr(B^2)}.
\]
We will use this to prove the result by showing that there is a $C$ such that for all $B \in W_{k}^\perp$, 
\begin{equation}
\label{r-eval-bound}
\| P_k (B) \|^2 
\geq 
\left( \frac{\lambda_{r+1}}{64\pi^3 k^2} + Ck^{-3}\right)
\tr(B^2).
\end{equation}
From Proposition \ref{hessian-balancing-identity}, we know that for any $B \in i\u(n_k+1)$,
\[
4\pi \| H(B) \|^2_{L^2(\omega_k)}
+
\| \xi_B^\top \|^2_{L^2(\omega_k)}
+
\|P_k(B)\|^2
=
\tr(B^2 \bar\mu_k).
\]
From Lemma \ref{bound-trABmu}, we deduce
\begin{equation}
\label{three-term-bound}
4\pi \| H(B) \|^2_{L^2(\omega_k)}
+
\| \xi_B^\top \|^2_{L^2(\omega_k)}
+
\| P_k(B)\|^2
\geq
\left(\frac{1}{4\pi} - Ck^{-1} \right) \tr(B^2)
\end{equation}
for some constant independent of $B$. We will prove the result by showing that when $B \in W_{k}^\perp$ the left-hand-side of (\ref{three-term-bound}) is bounded above by
\[
\left( \frac{16 \pi^2 k^2}{\lambda_{r+1}} + Ck \right)
\| P_k(B)\|^2
+
Ck^{-1} \tr(B^2)
\]
From here (\ref{r-eval-bound}) will follow.

To control the term $\| \xi_B^\top\|^2_{L^2(\omega_k)}$ we use a Lemma due to Phong and Sturm:

\begin{lemma}[Phong--Sturm \cite{phong-sturm}]
\label{phong-sturm-lemma}
Assume that $\Aut(X,L)/\C^*$ is discrete. There is a constant $C$ such that when ever $b \in \B_k$ has $R$-bounded geometry, then for all $B \in i\u(n_k+1)$, 
\[
\| \xi_B^\top \|^2_{L^2(\omega_k)}
\leq
C k \| P_k(B)\|^2.
\]
\end{lemma}

This is inequality (5.9) in \cite{phong-sturm}  and is proved in the course of the proof of Theorem 2 there. The sequence $\Hilb_k(h) \in \B_k$ that we are considering certainly has $R$-bounded geometry; thanks to  this result our work is reduced to showing that
\begin{equation}\label{hessian-and-tr-controls-HB}
\| H(B) \|^2_{L^2(\omega_k)}
\leq
\left( \frac{4 \pi k^2}{\lambda_{r+1}} + Ck \right)
\| P_k(B)\|^2
+
Ck^{-1} \tr(B^2).
\end{equation}

To do this, we work initially with the $L^2(\omega)$-norm. We partially decompose $H(B)$ in terms of the basis  $\phi_0, \phi_1, \ldots , \phi_{r}$,  introduced in \S\ref{lower-bound-proof-section}, for the span $E_r$ of the first $r+1$ eigen\-spaces of $\D^*\D$: 
\[
H(B)
=
\sum_{j=0}^{r}
\langle H(B) , \phi_j \rangle_{L^2(\omega)}\,\phi_j
+
\hat H
\]
where $\hat{H}$ is $L^2$-orthogonal to $E_r$. Now for each $j = 0, \ldots , r$, we apply Proposition~\ref{nu-espaces-converging} to the sequence $M_k = A_{j,k}$. (Recall that $A_{j,k}$ are the eigenvectors provided by the inductive hypotheses for which $H(A_{j,k})$ approaches $\phi_j$ and we write $W_k$ for their span; again, see \S\ref{lower-bound-proof-section}). The hypotheses of Proposition \ref{nu-espaces-converging} are satisfied since I3 gives that  
\[
\left\| H(A_{j,k}) - \phi_j\right \|^2_{L^2_2} = O(k^{-1})
\]
and I2 gives that $\tr(A_{j,k}^2) = 16\pi^2k^n + O(k^{n-1/2})$ (since the $\phi_j$ are of unit length in $L^2$). It follows that there is a constant $C$ such that if $B \in W_k^\perp$ then 
\[
\left| \langle H(B), \phi_j \rangle\right|^2
\leq 
Ck^{-n-1} \tr(B^2)
\]
for $j = 0, \ldots , r$. Hence,
\begin{equation}
\label{H-norm-in-decomp}
\| H(B) \|^2_{L^2(\omega)}
\leq
\| \hat H \|^2_{L^2(\omega)}
+ 
C k^{-n-1} \tr(B^2).
\end{equation}
Meanwhile, 
\begin{equation}
\label{DH-controls-hatH}
\| \hat H \|^2_{L^2(\omega)} \leq\frac{1}{\lambda_{r+1}} \| \D H(B) \|^2_{L^2(\omega)},
\end{equation}
since $\hat{H}$ is orthogonal to $E_{r}$ and so lies in the span of the eigenspaces of eigenvalue $\lambda \geq \lambda_{r+1}$. Applying Proposition \ref{P*P-controls-D}---which controls $\D H(B)$ in terms of the Hessian of balancing energy and the $L^2(\omega)$-norm of $H(B)$---we see that
\[
\| \hat H\|^2_{L^2(\omega)}
\leq
k^{2-n}\left(\frac{4\pi}{\lambda_{r+1}} + Ck^{-1}\right) 
\| P_k(B)\|^2
+
Ck^{-2} \| H(B) \|^2_{L^2(\omega)}
\]
Lemma \ref{tr-controls-L2H} gives that $\| H(B)\|^2_{L^2(\omega)} \leq C k^{-n} \tr(B^2)$, so that
\[
\| \hat H\|^2_{L^2(\omega)}
\leq
k^{2-n}\left(\frac{4\pi}{\lambda_{r+1}} + Ck^{-1}\right) 
\|P_k(B)\|^2
+
Ck^{-n-2} \tr(B^2).
\]
Combining this with (\ref{H-norm-in-decomp}) we see that
\[
\| H(B)\|^2_{L^2(\omega)}
	\leq
k^{2-n}\left(\frac{4\pi}{\lambda_{r+1}} + Ck^{-1}\right) 
\| P_k(B)\|^2
+
Ck^{-n-1} \tr(B^2)
\]

We now convert this bound to the $L^2(\omega_k)$-norm via Lemma \ref{L2k-L2-equivalence}, which gives a constant $C$ such that for any $f$, 
\[
\| f \|^2_{L^2(\omega_k)}
\leq
(k^n + Ck^{n-2}) \| f \|^2_{L^2(\omega)}.
\]
Applying this to $H(B)$ gives (\ref{hessian-and-tr-controls-HB}) which completes the proof.
\end{proof}

\section{Convergence of the eigenspaces}
\label{convergence-section}

By this stage, we have established I1 of \S\ref{strategy}; namely, if $\Aut(X,L)/\C^*$ is discrete, if $r^\text{th}$ inductive hypotheses hold, and if
\[
\lambda_r < \lambda_{r+1} = \cdots  = \lambda_s < \lambda_{s+1}
\]
then
\[
\nu_{j,k}
=
\frac{\lambda_r}{64\pi^3k^2} + O(k^{-3}).
\]
for $j = r+1, \ldots s$. (Proposition \ref{upper-bound-result} establishes the upper bound and Proposition~\ref{lower-bound-result} the lower bound in view of the trivial inequality $\nu_{j,k}\geq \nu_{r+1,k}$ for $j > r+1$.) To complete the induction, we must show that I2 and I3 hold.

\subsection{The proof of I2}

\begin{proposition}
\label{I2_holds}
Assume that $\Aut(X,L)/\C^*$ is discrete and that $\lambda_r < \lambda_{r+1} = \cdots  = \lambda_s < \lambda_{s+1}$. If the $r^{\text{th}}$-inductive hypotheses hold then there is a constant $C$ such that for all $A, B \in F_{s,k}$, 
\[
\left| \tr(AB) - 16\pi^2k^n\langle H(A), H(B) \rangle_{L^2(\omega)} \right|
\leq 
Ck^{-1}\tr(A^2)^{1/2}\tr(B^2)^{1/2}.
\]
\end{proposition}
\begin{proof}
Recall Proposition \ref{hessian-balancing-identity}, which says that
\[
\re(\tr(AB\bar\mu))
-
4\pi \langle H(A), H(B) \rangle_{L^2(\omega_k)}
=
\tr (A(P^*_kP_kB))
+
\re \langle \xi_A^\top, \xi_B^\top \rangle_{L^2(\omega_k)}.
\]
We consider each term in this equality separately. 

Since $A$ and $B$ lie in the span of eigenvectors with eigenvalue at most $\nu_{s,k}$ and since, by Proposition \ref{upper-bound-result}, $\nu_{s,k} = O(k^{-2})$, we have
\begin{equation}
\label{hessian_term}
\left|\tr (A(P^*_kP_kB))\right|
\leq 
Ck^{-2}\tr(A^2)^{1/2}\tr(B^2)^{1/2} 
\end{equation}
for some uniform constant $C$.

Next we apply Cauchy--Schwarz and Lemma \ref{phong-sturm-lemma} to deduce that 
\[
\left|\langle \xi_A^\top, \xi_B^\top \rangle_{L^2(\omega_k)}\right|
\ \leq\ 
\|\xi_A^\top\|_{L^2(\omega_k)}\|\xi_B^\top\|_{L^2(\omega_k)}
\ \leq\
Ck\|P_k(A)\|\,\|P_k(B)\|
\]
Again, as $A$ is an eigenvector of $P^*_kP_k$ with eigenvalue $O(k^{-2})$ we have $\|P_kA\|^2 \leq Ck^{-2}\tr(A^2)$ and similarly for $B$; hence 
\begin{equation}
\label{vector_term}
\left|\langle \xi_A^\top, \xi_B^\top \rangle_{L^2(\omega_k)}\right|
\leq
Ck^{-1} \tr(A^2)^{1/2}\tr(B^2)^{1/2}
\end{equation}
for some uniform $C$.

Now we use Lemma \ref{bound-trABmu}, which gives a uniform $C$ such that
\begin{multline}
\label{ABmubar_term}
\left|
\re\tr\left(AB\bar\mu_k\right)
-
\frac{1}{4\pi} \tr(AB)
\right|
\ \leq\
\left|
\tr\left(AB\bar\mu_k\right)
-
\frac{1}{4\pi} \tr(AB)
\right|\\
\ \leq\
Ck^{-1}\tr(A^2)^{1/2}\tr(B^2)^{1/2}
\end{multline}
(where the first inequality holds because $\tr(AB)$ is real).

Next, we use Lemma \ref{L2k-L2-equivalence}, which gives a uniform $C$ such that
\[
\left|
\langle H(A), H(B) \rangle_{L^2(\omega_k)}
-
k^n\langle H(A), H(B) \rangle_{L^2(\omega)}
\right|
\leq
Ck^{n-2} \| H(A)\|_{L^2(\omega)}\|H(B)\|_{L^2(\omega)}.
\]
Finally, Lemma \ref{tr-controls-L2H} means we can replace the norms of $H(A)$ and  $H(B)$ by those of $A$ and $B$ to get
\begin{equation}
\label{L2_term}
\left|
\langle H(A), H(B) \rangle_{L^2(\omega_k)}
-
k^n\langle H(A), H(B) \rangle_{L^2(\omega)}
\right|
\leq
Ck^{-2}\tr(A^2)^{1/2}\tr(B^2)^{1/2}
\end{equation}
Combining the inequalities (\ref{hessian_term}, \ref{vector_term}, \ref{ABmubar_term}, \ref{L2_term}) with the equality from Proposition \ref{hessian-balancing-identity} gives the estimate in the statement of the proposition.
\end{proof}

\subsection{The proof of I3}

To complete the proof of Theorems \ref{eigenvalue-theorem} and \ref{eigenvector-theorem} it remains to establish the last of the $s^\text{th}$-inductive hypotheses.

\begin{proposition}
\label{I3_holds}
Assume that $\Aut(X,L)/\C^*$ is discrete, that the $r^\text{th}$-inductive hypotheses hold and that $\lambda_r < \lambda_{r+1} = \cdots = \lambda_s < \lambda_{s+1}$. Given $\phi \in V_{r+1}$ a $\lambda_{r+1}$-eigenvector of $\D^*\D$, let $A_{\phi,k}$ be the point of $F_{r+1,s,k}$ for which $H(A_{\phi,k})$ is nearest to $\phi$ in $L^2$. Then
\[
\| H(A_{\phi,k}) - \phi \|^2_{L^2_2(\omega)} = O(k^{-1}).
\]
and this estimate is uniform in $\phi$ if in addition we require $\|\phi\|_{L^2} =1$.
\end{proposition}

To save repeatedly restating them, the hypotheses of Proposition \ref{I3_holds} are considered to hold throughout this section. We begin with a Lemma.

\begin{lemma}\label{L2-norm_DH}
Let $A$ be a $\nu_{j,k}$-eigenvector for $P^*_kP_k$ with $r+1 \leq j \leq s$. There is a uniform constant $C$ such that
\[
\left|
16\pi^2 k^n \left\| \D H(A)\right\|^2_{L^2(\omega)}
-
\lambda_{r+1} \tr(A^2) 
\right|
\leq C k^{-1} \tr(A^2).
\]
\end{lemma}

\begin{proof}
Since $A$ is a $\nu_{j,k}$-eigenvector, Proposition \ref{P*P-controls-D} gives
\[
\| \D H(A) \|^2_{L^2(\omega)}
\leq
\left(4\pi k^{2-n} + Ck^{1-n} \right)\nu_{j,k} \tr(A^2) 
+
Ck^{-2}\|H(A)\|^2_{L^2(\omega)}.
\]
Now Proposition \ref{upper-bound-result}  gives $\nu_{j,k} \leq \frac{\lambda_{r+1}}{64\pi^3k^2} + O(k^{-3})$ and Proposition \ref{I2_holds} gives 
\[
\left|16 \pi^2 k^n \| H(A) \|^2_{L^2(\omega)} - \tr(A^2)\right|
\leq Ck^{-1} \tr(A^2).
\]
So 
\[
16\pi^2k^n \| \D H(A) \|^2_{L^2(\omega)}
-
\lambda_{r+1}\tr(A^2)
\leq
Ck^{-1}\tr(A^2).
\]

To prove the lower bound, recall inequalities (\ref{H-norm-in-decomp}) and (\ref{DH-controls-hatH}), proved in the course of Proposition \ref{lower-bound-result}. These inequalities apply to any $B$ which is orthogonal to the span $F_{r,k}$ of eigenspaces of $P^*_kP_k$ with eigenvalue $\nu \leq \nu_{r,k}$ and so, in particular, to $A$. Combined, they show that there is a constant $C$ such that for any $\nu_{r+1,k}$-eigenvector $A$,
\[
\| \D H(A) \|^2_{L^2(\omega)}
\geq
\lambda_{r+1} \left(
\| H(A) \|^2_{L^2(\omega)} - Ck^{-n-1} \tr(A^2)
\right).
\]
Applying Proposition \ref{I2_holds} we see that 
\[
16\pi^2k^n\|\D H(A)\|^2_{L^2(\omega)}
-
\lambda_{r+1} \tr(A^2)
\geq
Ck^{-1} \tr(A^2).
\]
\end{proof}

\begin{lemma}
\label{L22-L2-equivalent}
There is a constant $C$ such that if $A$ is a $\nu_{j,k}$-eigenvector of $P^*_kP_k$ with $j \leq s$, then
\[
\| H(A)\|_{L^2_2(\omega)} \leq C \| H(A) \|_{L^2(\omega)}
\]
\end{lemma}
\begin{proof}
We first recall that since $\D$ is overdetermined elliptic (i.e., its symbol is injective) there is a constant $C$ such that for any $f \in L^2_2$,
\[
\| f\|_{L^2_2(\omega)} \leq C \left (\| f\|_{L^2(\omega)} + \| \D f\|_{L^2(\omega)} \right).
\] 
By Proposition \ref{I2_holds}, if we pick some $c > 16\pi^2$ then for all $k$ sufficiently large, 
\[
\tr(A^2) 
\leq 
c k^n \| H(A)\|^2_{L^2(\omega)}.
\]
Now from this and Lemma \ref{L2-norm_DH}, we have
\[
\| \D H(A)\|^2_{L^2(\omega)}
\leq
C \| H(A)\|^2_{L^2(\omega)}.
\]
(Strictly speaking Lemma \ref{L2-norm_DH} applies for $r+1 \leq j \leq s$, we must also invoke the identical versions for smaller values of $j$.) This and the elliptic estimate prove the result. 
\end{proof}

Now given $A \in F_{r+1,s,k}$ we write
\[
H(A) = H(A)_< + H(A)_{r+1} + H(A)_>
\]
where $H(A)_<$ is the component of $H(A)$ lying in the span $E_r$ of all $\D^*\D$-eigen\-spaces with eigenvalue strictly less than $\lambda_{r+1}$, $H(A)_>$ lies in the span of eigen\-spaces with eigenvalue strictly greater than $\lambda_{r+1}$ and $H(A)_{r+1}$ is the component of $H(A)$ in the $\lambda_{r+1}$-eigenspace $V_{r+1}$. We next show that $H(A)_{r+1}$ is the dominant part in this decomposition.

\begin{lemma}\label{HA-r+1_bound}
There is a constant $C$ such that for all $A \in F_{r+1,s, k}$
\[
\left| 16\pi^2k^n\|H(A)_{r+1}\|^2_{L^2}  - \tr(A^2)\right|  
\leq Ck^{-1}\tr(A^2).
\]
\end{lemma}
\begin{proof}
By rescaling it suffices to prove the result for $\tr(A^2) = 16\pi^2k^n$, which will simplify the notation. We know from Proposition \ref{I2_holds} and Lemma \ref{L2-norm_DH} that for such~$A$,
\begin{eqnarray*}
\|H(A)_<\|^2_{L^2(\omega)}
+
\|H(A)_{r+1}\|^2_{L^2(\omega)} 
+
\|H(A)_>\|^2_{L^2(\omega)}
	&=&
		1 + O(k^{-1}),\\
\|\D H(A)_<\|^2_{L^2(\omega)}
+
\lambda_{r+1}\|H(A)_{r+1}\|^2_{L^2(\omega)} 
+
\|\D H(A)_>\|^2_{L^2(\omega)}
	&=&
		\lambda_{r+1} +O(k^{-1}).
\end{eqnarray*}
Now let $\phi_0, \ldots, \phi_r$ be an orthonormal basis for $E_r$ with $\phi_j$ a $\lambda_j$-eigenvector of $\D^*\D$. By our inductive hypotheses, there are $P_k^*P_k$-eigenvectors $A_{j,k}$ all with eigenvalues at most $\nu_{r,k}$, with $\tr(A_{j,k}) = 16\pi^2k^{n} + O(k^{n-1})$ and with $H(A_{j,k}) = \phi_j +O(k^{-1/2})$ in $L^2_2$. This means we can apply  Proposition \ref{nu-espaces-converging} to $A$, which is orthogonal to each $A_{j,k}$, to obtain
\[
\left|\langle H(A), \phi_j \rangle_{L^2(\omega)}\right|^2
\ \leq \
C k^{-n-1} \tr(A^2)
\ \leq \
C k^{-1}.
\]
It follows that $\| H(A)_{<}\|^2_{L^2(\omega)} \leq Ck^{-1}$ and so also
\[
\| \D H(A)_<\|^2_{L^2(\omega)}
\  \leq \
\lambda_r \| H(A)_{<}\|^2_{L^2(\omega)}
\ \leq\ Ck^{-1}.
\]
Meanwhile, $\| \D H(A)_{>} \|^2_{L^2(\omega)} \geq \lambda_{s+1}\|H(A)_>\|^2_{L^2(\omega)}$. 

These bounds now give us
\begin{eqnarray*}
\| H(A)_{r+1}\|^2_{L^2(\omega)} + \|H(A)_>\|^2_{L^2(\omega)}
	&=&
		1 + O(k^{-1}),\\
\lambda_{r+1}\|H(A)_{r+1}\|^2_{L^2(\omega)} + \lambda_{s+1}\|H(A)_>\|^2_{L^2(\omega)}
	&\leq&
		\lambda_{r+1} + O(k^{-1})
\end{eqnarray*}
Subtract $\lambda_{r+1}$-times the first estimate from the second. Since $\lambda_{s+1} > \lambda_{r+1}$ we see that $\|H(A)_>\|^2_{L^2(\omega)} = O(k^{-1})$ and so $\|H(A)_{r+1}\|^2_{L^2(\omega)} = 1 + O(k^{-1})$, which is what the Lemma asserts for $\tr(A^2) = 16\pi^2k^n$. 		
\end{proof}

We can now complete the proof of Proposition \ref{I3_holds} and thus those of Theorems \ref{eigenvalue-theorem} and \ref{eigenvector-theorem}. 

\begin{proof}[Proof of Proposition \ref{I3_holds}] Let $\phi \in V_{r+1}$ be of unit length in $L^2$. Denote by $A_{\phi,k} \in F_{r+1,s,k}$ the point for which $H(A)$ is nearest to $\phi$ in $L^2$. We first note that the map $A \mapsto H(A)_{r+1}$ is a linear isomorphism $F_{r+1,s,k} \to V_r$. This is because, by definition, $\dim F_{r+1,s,k} \geq s-r = \dim V_r$; meanwhile by Lemma~\ref{HA-r+1_bound}, if $H(A)_{r+1} = 0$, then $A=0$, so the map is injective. 

It follows that $A_{\phi,k}$ is the unique element of $F_{r+1,s,k}$ for which $H(A)_{r+1} = \phi$:
\[
H(A_{\phi,k}) = H(A_{\phi,k})_{<} + \phi + H(A_{\phi,k})_{>}
\]
Moreover, since $\|\phi\|_{L^2} = 1$, Lemma \ref{HA-r+1_bound} tells us that $\tr(A^2_{\phi,k}) = 16\pi^2k^n + O(k^{n-1})$. Now Proposition \ref{I2_holds} gives that $\|H(A_{\phi,k})\|^2_{L^2} = 1 + O(k^{-1})$ from which we deduce that 
\[
\| H(A_{\phi,k}) - \phi \|^2_{L^2}
=
\|H(A_{\phi,k})_<\|^2_{L^2} + \|H(A_{\phi,k})_>\|^2_{L^2}
=
O(k^{-1})
\]
and the convergence is seen to hold in $L^2$. 

To extend this to $L^2_2$ we note that by Lemma \ref{L22-L2-equivalent}, the $L^2_2$-norm and $L^2$-norms are uniformly equivalent for the functions $H(A_{\phi,k})$ under consideration. This implies that $H(A_{\phi,k})$ is Cauchy and hence convergent in $L^2_2$. It remains to check the rate of convergence in $L^2_2$. To show this, note that by Lemma \ref{L2-norm_DH} and the fact that 
 $\tr(A^2_{\phi,k}) = 16\pi^2k^n + O(k^{n-1})$ we have
$
\| \D H(A_{\phi,k}) \|^2_{L^2} = \lambda_{r+1} + O(k^{-1})
$
and hence $\|\D H(A_{\phi,k})_<\|^2_{L^2} + \| \D H(A_{\phi,k})_>\|^2 = O(k^{-1})$; in other words, $\|\D (H(A_{\phi,k}) - \phi)\|_{L^2}^2 = O(k^{-1})$. The proof that $\|H(A_{\phi,k}) - \phi\|^2_{L^2_2} = O(k^{-1})$ now proceeds via the elliptic estimate for $\D$ just as in the proof of Lemma \ref{L22-L2-equivalent}.

Finally, we observe that the estimate here is uniform in $\phi$ since at no point did we make use of anything other than $\|\phi\|_{L^2} = 1$ and $\phi \in V_{r+1}$.
\end{proof}

\section{The Hessians for balanced embeddings}
\label{Hessians_of_balanced} 

In this final section we prove that the convergence results of Theorems \ref{asymptotics-hessians}, \ref{eigenvalue-theorem} and \ref{eigenvector-theorem} apply not just to a sequence of the form $\Hilb_k(h)$ for some fixed $h$, but also to a sequence of balanced embeddings which converge to a constant scalar curvature Kähler metric, as in Donaldson's Theorem \ref{skd_balanced_converges}. This was stated as Theorem~\ref{convergence_for_balanced} in \S\ref{applications}. We warn the reader that the notation in this section---in particular the meaning of a $k$-subscript---is a little different to that used in the body of the article. This is mentioned explicitly in what follows.
 
\begin{proof}[Proof of Theorem \ref{convergence_for_balanced}]
Recall that we assume $\Aut(X,L)/\C^*$ is discrete, that $\omega_{\mathrm{csc}} \in c_1(L)$ has constant scalar curvature, that $b_k \in \B_k$ is balanced and that the curvatures $\omega_k \in c_1(L)$ of $h_k = \FS(b_k)$ converge in $C^\infty$ to $\omega_{\mathrm{csc}}$, as in Theorem \ref{skd_balanced_converges}. In fact, the proof of Theorem 1 gives $\omega_k = \omega_{csc} + O(k^{-1})$ in $C^r$ for any $r$ (where the $C^r$-norm is defined via $\omega_{csc}$). Now, for each large integer $l$, we have a sequence $b_{k,l} = \Hilb_k(h_l)$ of projective metrics. The key fact we will use is that, since $b_k$ is balanced, it is a fixed point of $\Hilb_k\circ \FS_k$ (see \cite{donaldson-1}). In other words, the diagonal sequence $b_{k,k}= b_k$ is the original sequence of balanced embeddings.

We begin with the proof of part 1, that for $f, g \in C^\infty(X)$, 
\[
\tr\left(Q_{f,k}P^*_kP_k\left(Q_{g,k}\right)\right)
=
\frac{k^n}{4\pi} \int_X f \D^*\D g\, \frac{\omega^n_{\mathrm{csc}}}{n!}
+
O(k^{n-1})
\]
where \emph{the subscript $k$ means the object is defined with respect to the balanced embedding $b_k$}.  To see this we apply Theorem \ref{asymptotics-hessians} to each term in the sequence $h_l$. Write $Q_{f,k,l}$ for the matrix defined by  equation (\ref{Q-matrix}) where the basis $s_\alpha$ of $H^0(X,L^k)$ is orthonormal for the metric $b_{k,l}$. By the balanced property mentioned above, $Q_{f,k,k} = Q_{f,k}$ is the matrix appearing in the statement of part~1. Similarly, we write $P_{k,l}$ for the operator defined by the embedding $b_{k,l}$ and again $P_{k,k} = P_k$ is defined by the balanced embedding $b_k$. Finally, we write $\D^*_l\D_l$ for the operator defined by the Kähler metric $\omega_l$.

Now Theorem \ref{asymptotics-hessians} gives, for each large integer $l$, a constant $C$ such that
\[
\left|
\tr\left(Q_{f,k,l} P^*_{k,l}P_{k,l}(Q_{g,k,l})\right)
-
\frac{k^n}{4\pi}\int_X f \D_l^*\D_l g\, \frac{\omega^n_{\mathrm{l}}}{n!}
\right|
\leq 
Ck^{n-1}.
\]
Since $\omega_l \to \omega_{\mathrm{csc}}$ in $C^\infty$, the $\omega_l$ are uniformly equivalent and form a family which is compact for the $C^\infty$-topology. It follows that $C$ in the above estimate can be chosen independent of $l$. Setting $l = k$ we obtain
\[
\left|
\tr\left(Q_{f,k} P^*_{k}P_{k}(Q_{g,k})\right)
-
\frac{k^n}{4\pi}\int_X f \D_k^*\D_k g\, \frac{\omega^n_{\mathrm{k}}}{n!}
\right|
\leq 
Ck^{n-1}.
\]
Since $\omega_k = \omega_{\mathrm{csc}} + O(k^{-1})$ in $C^r$ for all $r$, we have that
\[
\int_X f \D^*_k\D_k g\, \frac{\omega_k^n}{n!} 
= 
\int_X f \D^* \D g\, \frac{\omega_{\mathrm{csc}}^n}{n!}
+
O(k^{-1})
\]
and part 1 now follows.

The proofs of parts 2 and 3 follow exactly the same argument. In part 2, for example, along with the uniformity of Theorem \ref{eigenvalue-theorem} we use the fact that the eigenvalues of $\D^*\D$ depend continuously on the metric. We should just point out that to obtain a more precise statement here about the rate of convergence of $64\pi^3 k^2 \nu_{k,j}$ to $\lambda_j$, one would need a quantitative measure of this continuity. 

Finally we give the proof of part 4 which is again similar, but the wording needs a little extra care. Recall that $\lambda_{p-1} < \lambda_p = \cdots = \lambda_q < \lambda_{q+1}$, that $\phi$ is a $\lambda_p$-eigenvector of $\D^*\D$ and that $A_k\in F_{p,q,k}$ is the $\nu_{j,k}$-eigenvector of $P^*_kP_k$ with $p \leq j \leq q$ and with $H(A)$ nearest to $\phi$ in $L^2$. We must show that $H(A_k)$ converges to $\phi$ in $L^2_2$. 

To begin, we prove convergence in $L^2$. First note that since $\omega_k \to \omega_{\mathrm{csc}}$ in $C^\infty$, there is a sequence $\phi_k$ converging to $\phi$ in $L^2$ and such that $\phi_k$ is a $\D^*_k\D_k$-eigenvector with eigenvalue $\lambda_j(\omega_k)$ where $p \leq j \leq q$. Now the uniformity of part~2 of Theorem~\ref{eigenvector-theorem} together with a diagonal argument as above gives the existence of a sequence $A'_k \in i \u(n_k+1)$ of $\nu_{i,k}$-eigenvectors with $p \leq i \leq q$ and a constant $C$ such that 
\[
\left\|H(A'_k) - \phi_k\right\|^2_{L^2_2(\omega_k)} 
< Ck^{-1}
\] 
$A'_k$ is defined as follows: suppose that $\lambda_{p'-1}(\omega_k) < \lambda_{p'}(\omega_k) = \cdots = \lambda_{q'}(\omega_k) < \lambda_{q'+1}(\omega_k)$ and that $p' \leq i \leq q'$. Then $A'_k$ is the point of $F_{p',q',k}$ for which $H(A'_k)$ is closest to $\phi_k$ in $L^2(\omega_k)$. In particular, $A'_k$ is \emph{not} the same as $A_k$.

Since $\omega_k \to \omega_{\mathrm{csc}}$ in $C^\infty$ we can replace $\omega_k$ by $\omega_{\mathrm{csc}}$ in the $L^2_2$-norm here. It follows that $H(A'_k)$ converges to $\phi$ in $L^2(\omega_{\mathrm{csc}})$. Now since $H(A_k)$ is, by definition of $A_k$, closer in $L^2(\omega_{\mathrm{csc}})$ to $\phi$ than $H(A'_k)$ we see that $H(A_k)$ also converges in $L^2$ to $\phi$. 

To complete the proof we must extend the convergence to $L^2_2$. This follows essentially from Lemma \ref{L22-L2-equivalent}, which says that the $L^2_2$-norm and $L^2$-norm are equivalent for the functions $H(A_k)$. To run this argument, we apply the Lemma to each of the metrics $\omega_l$, use uniformity of the constant with respect to $l$ (cf.\ the remark of \S\ref{uniformity_metric}) and then apply a diagonal argument. This gives a constant $C$ such that for all $A \in F_{p,q,k}$,
\[
\| H(A) \|_{L^2_2(\omega_k)} \leq C \| H(A) \|_{L^2(\omega_k)}.
\]
We now replace the $\omega_k$-norms by the uniformly equivalent $\omega_{\mathrm{csc}}$-norms to finish the argument.
\end{proof}

\bibliographystyle{alpha}
\bibliography{quant_hess_mabuchi_v6}

\end{document}